\documentclass[11pt,a4paper,reqno]{amsart}
\usepackage{amssymb,amsmath,amsthm,graphicx,color}
\theoremstyle{plain}
\newtheorem{theorem}{Theorem}[section]
\newtheorem{proposition}[theorem]{Proposition}

\newtheorem{lemma}[theorem]{Lemma}

\theoremstyle{remark}

\usepackage
[bookmarks=true,
   bookmarksnumbered=false,
   bookmarkstype=toc]
   {hyperref}

\numberwithin{equation}{section}

\newcommand{\C}{\mathbb{C}}
\newcommand{\R}{\mathbb{R}}

\def\({\left(}
\def\){\right)}
\def\<{\left\langle}
\def\>{\right\rangle}
\def\le{\leqslant}
\def\ge{\geqslant}

\newcommand{\rre}{\mathbb{R}}
\newcommand{\pt}{\partial}

\begin{document}
\title[Schr\"{o}dinger equation with higher order dispersion]
{Asymptotic behavior in time of solution\\
to the nonlinear Schr\"{o}dinger equation\\
with higher order anisotropic dispersion}

\author[J.-C. Saut]{Jean-Claude Saut}

\address{Laboratoire de Math\' ematiques, CNRS and Universit\' e Paris-Sud\\
91405 Orsay, France}
\email{jean-claude.saut@u-psud.fr}
\author[J.Segata]
{Jun-ichi Segata}

\address{Mathematical Institute, Tohoku University\\
6-3, Aoba, Aramaki, Aoba-ku, Sendai 980-8578, Japan}
\email{segata@m.tohoku.ac.jp}




\maketitle

\begin{abstract}
We consider the asymptotic behavior in time of solutions 
to the nonlinear Schr\"{o}dinger equation with fourth order anisotropic 
dispersion (4NLS) which describes the propagation
of ultrashort laser  pulses  in a 
medium  with  anomalous  time-dispersion in  the  presence  of
fourth-order time-dispersion. We prove existence of 
a solution to (4NLS)
which scatters to a solution of the linearized equation 
of (4NLS) as $t\to\infty$.

\end{abstract}

\section{Introduction}

We consider the asymptotic behavior in time of solution 
to the nonlinear Schr\"{o}dinger equation with higher order anisotropic 
dispersion:  
\begin{eqnarray}
i\pt_tu+\alpha\Delta u+i\beta\pt_{x_{1}}^{3}u+\gamma\pt_{x_{1}}^4u=
\lambda|u|^{p-1}u,
\quad t>0,x\in\rre^{d},
\label{D}
\end{eqnarray}
where $u:\rre\times\rre^{d}\to\C$ is an unknown function, 
$\alpha,\beta,\gamma,\lambda$ 
are real constants and $p>1$. Equation (\ref{D}) 
arises in nonlinear optics to model the propagation
of ultrashort laser  pulses  in a 
medium  with  anomalous  time-dispersion  in  the  presence  of
fourth-order time-dispersion   (see \cite {WF, FIS, Ber} 
and the references therein). It also arises  in models of propagation
 in fiber arrays (see \cite{ADRT, FIS1}).

To simplify (\ref{D}), we introduce a new unknown function 
\begin{eqnarray*}
v(t,x)=e^{-i(\frac{\alpha\beta^{2}}{16\gamma^{2}}
+\frac{5\beta^{4}}{256\gamma^{3}})t+i\frac{\beta}{4\gamma}x_{1}}
u(t,x_{1}-(\frac{\alpha\beta}{2\gamma}
+\frac{\beta^{3}}{8\gamma^{2}})t,x_{\perp}),
\end{eqnarray*}
where $x=(x_{1},x_{\perp})\in\rre\times\rre^{d-1}$. 
Then the equation (\ref{D}) can be rewritten as 
\begin{eqnarray*}
i\pt_tv+
\{(\alpha+\frac{3\beta^{2}}{8\gamma})\pt_{x_{1}}^2
+\alpha\Delta_{\perp}\}v+\gamma\pt_{x_{1}}^4v=\lambda|v|^{p-1}v,
\qquad t>0,x\in\rre^{d}.
\end{eqnarray*}
Therefore if $\alpha,\beta$ and $\gamma$ satisfy $(\alpha+\frac{3\beta^{2}}{8\gamma})\alpha>0$ 
and $\alpha\gamma<0$, then by using a suitable scaling transform, we can rewrite (\ref{D}) 
into the Schr\"{o}dinger equation 
with fourth order anisotropic dispersion:
\begin{eqnarray}
i\pt_tu+\frac12\Delta u-\frac{1}{4}\pt_{x_{1}}^4u=\lambda|u|^{p-1}u,
\qquad t>0,x\in\rre^{d}.
\label{4NLS}
\end{eqnarray}
In this paper, we study the asymptotic behavior in time of solutions 
to (\ref{4NLS}).

The Cauchy problem for the  homogeneous fourth order nonlinear Schr\"odinger type equation 
\begin{eqnarray}
i\pt_tu+\frac12\Delta^{2} u=\lambda|u|^{p-1}u,
\qquad t>0,x\in\rre^{d}
\label{h}
\end{eqnarray}
has been  studied by many authors,  most of the results holding also when  lower dispersive terms are added.  By using the Strichartz estimates in \cite{BKS} one shows  that the 
Cauchy problem is locally well-posed in the energy space $H^2(\R^d)$ for the energy subcritical case (i.e., 
$1<p<1+8/(d-4)$ when $d\geq 5$ and $1<p<\infty$ when $d\leq 4$) and in $L^2(\R^d)$ for the mass subcritical case ($
1<p<1+8/d$). We also  refer  to Bouchel \cite{B} who studies the Cauchy problem and furthermore  gives non-existence, existence and qualitative properties results of solitary wave solutions for (1.1). See also \cite{FIS} for results on the Cauchy problem for slightly more general situations.

There are 
several results concerning the scattering and blow-up of solutions 
for (\ref{h}).
For the defocusing case $\lambda<0$, 
the global well-posedness and scattering for (\ref{h}) 
with the energy-critical nonlinearity (i.e., (\ref{h}) 
with $d\ge5,$ and $p=1+8/(d-4)$) was studied by 
Pausader \cite{P1} for radially symmetric initial data  
by combining the concentration-compactness argument by Kenig-Merle \cite{KM} 
and Morawetz-type estimate. 
Later, Miao, Xu and Zhao \cite{MXZ2} 
proved a similar result for (\ref{h}) 
in the energy-critical and higher dimensional case 
$d\ge9$ without radial 
assumption on initial data. In \cite{P2}, Pausader 
has shown the global well-posedness and scattering of  
(\ref{h}) with cubic nonlinearity for the case $5\le d\le8$. 
Pausader and Xia \cite{PX} proved the global well-posedness 
and scattering for (\ref{h}) 
with  mass super-critical nonlinearity (i.e., (\ref{h}) 
with $p>1+8/d$) for low dimensions $1\le d\le 4$ 
by using a virial-type estimate instead of 
the Morawetz-type estimates. 

For the focusing case $\lambda>0$, 
Pausader \cite{P0} and Miao, Xu and Zhao \cite{MXZ1} 
independently showed the global well-posedness and scattering for (\ref{h}) 
with the energy-critical nonlinearity for  radially 
symmetric initial data with $\dot{H}^2$ and energy norms below  that of the ground state. 
When the initial data is sufficiently small, 
Hayashi, Mendez-Navarro and Naumkin \cite{HN7} 
proved the global existence and the scattering 
for (\ref{h}) with $d=1$ and $p>5$ by using the factorization technique 
developed by the authors \cite{HN3}. In \cite{HN7} 
they  also shown the small data global existence and 
the decay estimates for (\ref{h}) with $d=1$ and $p>4$ 
under the assumption that the initial data is odd. 
In the subsequent paper \cite{HN8}, 
they proved that when $d=1,p=5$ and $\lambda<0$, 
a solution to (\ref{h}) has dissipative structure 
and gains additional logarithmic decay. 
Aoki, Hayashi and Naumkin \cite{AHN} 
showed the global existence and scattering 
of (\ref{h}) with $d=1,2$ and $p>1+4/d$. 
We refer  also to the series of paper by Hayashi and Naumkin 
\cite{HN4,HN5,HN6} and work by Hirayama and Okamoto \cite{HO} 
for interesting phenomena on 
the long time behavior of solutions to (\ref{h}) 
with a derivative nonlinearity. 

Recently, a blow-up result is proved by Boulenger and Lenzmann \cite{BL} 
for (\ref{h}) with the mass critical and super critical 
focusing nonlinearity in the radial case which solves a 
long standing conjecture suggested by several numerical studies 
(see \cite{FIS1} for instance). Notice that most of their results hold also 
when lower dispersive term $\mu\Delta u$ is added. 
See also the work by Bonheure, Casteras, Gou and Jeanjean 
\cite{BCGJ} for the extension of the blow up results by \cite{BL}.

We now return to the inhomogeneous case. 
Since the point-wise decay of the solution 
to the linear fourth order Schr\"{o}dinger equation
\begin{eqnarray}
i\pt_tu+\frac12\Delta u-\frac{1}{4}\pt_{x_{1}}^4u=0,
\qquad t>0,x\in\rre^{d}
\label{4L}
\end{eqnarray}
is 
${\mathcal O}(t^{-d/2})$ as $t\to\infty$ (see Ben-Artzi, Koch and Saut 
\cite{BKS}), 
we expect that if $p>1+2/d$, then the (small) solution 
to (\ref{4NLS}) will scatter to the solution to the linearized equation (\ref{4L}). 
Compared to the homogeneous equation (\ref{h}), 
there are few results on the long time behavior of 
solution for (\ref{4NLS}).   
For the one dimensional cubic case, Segata \cite{S} proved that 
for a given asymptotic profile, there exists a solution $u$ to (\ref{4NLS}) 
which converges to the given asymptotic profile as $t\to\infty$, 
where the asymptotic profile is given by the leading term of the solution 
to the linear equation (\ref{4L}) with a logarithmic phase correction. 
Furthermore, Hayashi and Naumkin \cite{HN9} proved that for any small 
initial data, there exists a global solution to (\ref{4NLS}) with $d=1,p=3$ 
which behaves like a solution to the linear equation (\ref{4L}) 
with a logarithmic phase correction.
 
In this paper, we consider the small data scettering of (\ref{4NLS}) 
for the higher dimensional case. Notice that the interesting case 
is for $p\in(1+2/d,1+4/d)$ since for the case $p\ge1+4/d$, the nonlinearity 
in (\ref{4NLS}) is weaker compared to the case $p<1+4/d$. 
We shall show that for $d=2,3$ and some range of $p$ in $(1+2/d,1+4/d)$, 
there exists a solution to (\ref{4NLS}) which scatters to the solution to the 
linearized equation (\ref{4L}). Let us consider the final state problem:
\begin{eqnarray}
\left\{
\begin{array}{l}
\displaystyle{
i\pt_tu+\frac12\Delta u-\frac{1}{4}\pt_{x_{1}}^4u=\lambda|u|^{p-1}u,
\qquad t>0,x\in\rre^{d},}\\
\displaystyle{\lim_{t\to+\infty}(u(t)-W(t)\psi_{+})=0,\qquad in\ L^2,}
\end{array}
\right.
\label{FSP}
\end{eqnarray}
where $\{W(t)\}_{t\in\rre}$ 
is a unitary group generated by the operator $(1/2)i\Delta-(1/4)i\pt_{x_{1}}^4$, 
and $\psi_+$ is a given function.

Our main results in this paper are as follows:

\begin{theorem}\label{nonlinear1}
Let $d=2$ and let $2<p<3$. 
Then 
for any $\psi_+\in H^{0,s}(\R^2)$ 
with $(p+1+2\varepsilon)/2<s<p$  
and sufficiently small number $\varepsilon>0$, 
there exists a unique global solution $u\in C(\rre;L_x^2(\rre^{2}))
\cap \langle\pt_{x_{1}}\rangle^{-1/q}L_{loc}^{q}(\rre;L_x^{r}(\rre^{2}))$ 
to (\ref{FSP}) satisfying 
\begin{eqnarray*}
\|u-W(t)\psi_+\|_{L^{\infty}(t,\infty;L_x^2)}
\le Ct^{-\gamma}
\end{eqnarray*}
for $t\ge3$, where $\gamma=\min\{1/(p-1)-1/2,2-p,3/8\}$, 
and 
\begin{eqnarray*}
\left(q,r\right)
=\left(\frac{2}{p-2+2\varepsilon},\frac{2}{3-p-2\varepsilon}\right). 
\end{eqnarray*}
\end{theorem}

\begin{theorem}\label{nonlinear2}
Let $d=3$ and let $9/5<p<7/3$. 
Then 
for any $\psi_+\in H^{0,s}(\R^3)$ 
with $p-1/6<s<p$, 
there exists a unique global solution $u\in C(\rre;L_x^2(\rre^{3}))
\cap \langle\pt_{x_{1}}\rangle^{-2/(3q)}L_{loc}^{q}(\rre;L_x^{r}(\rre^{3}))$ 
to (\ref{FSP}) satisfying 
\begin{eqnarray*}
\|u-W(t)\psi_+\|_{L^{\infty}(t,\infty;L_x^2)}
\le Ct^{-\gamma}
\end{eqnarray*}
for $t\ge3$, where 
$\gamma=\min\{1/(p-1)-3/4,(5-3p)/2,3/8\}$,  
and 
\begin{eqnarray*}
\left(q,r\right)
=\left(\frac{4}{3p-5},\frac{6}{8-3p}\right). 
\end{eqnarray*}
\end{theorem}

To prove Theorems \ref{nonlinear1} and \ref{nonlinear2}, 
we employ the argument by Ozawa \cite{O}, Hayashi and Naumkin \cite{HN3,HN1}. 
Let us introduce a modified asymptotic profile 
\begin{equation}\label{u}
\begin{aligned}
\quad u_+(t,x)&=
\frac{t^{-\frac{d}{2}}}{\sqrt{3\mu_{1}^{2}+1}}\hat{\psi}_{+}(\mu)
e^{\frac34it\mu_{1}^{4}+\frac12it|\mu|^{2}
+iS_{+}(t,\mu)-i\frac{d}{4}\pi},\\
S_+(t,\xi)&=-\frac{\lambda}{1-\frac{p-1}{2}d}
\frac{|\hat{\psi}_{+}(\xi)|^{p-1}}
{(3\xi_{1}^{2}+1)^{\frac{p-1}{2}}}t^{1-\frac{p-1}{2}d},
\end{aligned}
\end{equation}
where $\mu=(\mu_{1},\mu_{\perp})$ is given by 
\begin{equation}\label{mu}
\begin{aligned}
\mu_{1}&=
\left\{\frac1{2t}\left(x+\sqrt{x^{2}+\frac{4}{27}t^{2}}\right)
\right\}^{1/3}
+
\left\{\frac1{2t}\left(x-\sqrt{x^{2}+\frac{4}{27}t^{2}}\right)
\right\}^{1/3},\\
\mu_{\perp}&=\frac{x_{\perp}}{t}.
\end{aligned}
\end{equation}
Notice that the right hand side of (\ref{u}) with $S_+\equiv0$ is the 
leading term of the solution to the linear fourth order equation (\ref{4L}).  
We first construct a solution $u$ to the final state problem 
\begin{eqnarray}
\left\{
\begin{array}{l}
\displaystyle{
i\pt_tu+\frac12\Delta u-\frac{1}{4}\pt_{x_{1}}^4u=\lambda|u|^{p-1}u,
\qquad t>0,x\in\rre^{d},}\\
\displaystyle{\lim_{t\to+\infty}(u(t)-u_{+})=0,\qquad in\ L^2.}
\end{array}
\right.
\label{FSP2}
\end{eqnarray}
To prove this, 
we first rewrite (\ref{FSP2}) as the integral equation 
\begin{eqnarray}
u(t)-u_+(t)&=&
i\lambda\int_t^{+\infty}
W(t-\tau)[|u|^{p-1}u-|u_+|^{p-1}u_+](\tau)d\tau\nonumber\\
& &+R_1(t)-i\int_t^{+\infty}W(t-\tau)R_2(\tau)d\tau,
\label{MINT}
\end{eqnarray}
where
\begin{eqnarray*}
R_1(t,x)
&=&W(t){{\mathcal F}}^{-1}
[\hat{\psi}_+(\xi)e^{iS(t,\xi)}](x)
-u_+(t,x),\\
R_2(t,x)
&=&W(t){{\mathcal F}}^{-1}[
\lambda
\frac{t^{-\frac{p-1}{2}d}}
{(3\xi_{1}^{2}+1)^{\frac{p-1}{2}}}|\hat{\psi}_+|^{p-1}
\hat{\psi}_+(\xi)e^{iS(t,\xi)}](x)\\
& &-\lambda|u_+|^{p-1}u_+(t,x).
\end{eqnarray*}
For the derivation of (\ref{MINT}), see Section 4 below. 
Next, we apply the contraction mapping principle to 
the integral equation (\ref{MINT})  
in a suitable function space. 
In this step the asymptotic formula (Proposition \ref{linear}) 
and the Strichartz estimate (Lemma \ref{S}) for 
the linear equation (\ref{4L}) play an important role. 
Finally, we show that the solutions of (\ref{FSP2}) 
converge to $W(t)\psi_{+}$ in $L^{2}$ as $t\to\infty$. 

We introduce several notations and function spaces 
which are used throughout this paper. 
For $\psi\in{{\mathcal S}}'(\rre^{d})$, 
$\hat{\psi}(\xi)={{\mathcal F}}[\psi](\xi)$ 
denote the Fourier transform of $\psi$. Let 
$\langle\xi\rangle=\sqrt{|\xi|^2+1}$. The differential operators
$|\nabla|^s=(-\Delta)^{s/2}$ and $\langle\nabla\rangle^s=(1-\Delta)^{s/2}$ denote 
the Riesz potential and Bessel potential of order $-s$, respectively. 
We define $\langle\pt_{x_{1}}\rangle^{s}={{\mathcal F}}^{-1}
\langle\xi_{1}\rangle^{s}{{\mathcal F}}^{-1}$. 
For $1\le q,r\le\infty$, $L^q(t,\infty;L_x^r(\rre^{d}))$ is defined 
as follows:
\begin{eqnarray*}
L^q(t,\infty;L_x^r(\rre^{d}))&=&\{u\in{{\mathcal S}}'(\rre^{1+d});
\|u\|_{L^q(t,\infty;L_x^r)}<\infty\},\\
\|u\|_{L^q(t,\infty;L_x^r)}&=&
\left(\int_t^{\infty}\|u(\tau)\|_{L_x^r}^qd\tau\right)^{1/q}.
\end{eqnarray*}
We will use the  Sobolev spaces
\begin{eqnarray*}
H^s(\R^d)=\lbrace \phi \in \mathcal S'(\R^d); 
\|\phi\|_{H^s}=\| 
\langle \nabla\rangle^{s} \phi\|_{L^2}<\infty\rbrace
\end{eqnarray*}
and their homogeneous version 
\begin{eqnarray*}
\dot{H}^s(\R^d)=\lbrace \phi \in \mathcal S'(\R^d); 
\|\phi\|_{\dot{H}^s}=\| 
|\nabla|^{s} \phi\|_{L^2}<\infty\rbrace.
\end{eqnarray*}
The weighted Sobolev spaces is defined by 
\begin{eqnarray*}
H^{m,s}(\R^d)=\lbrace \phi \in \mathcal S'(\R^d); 
\|\phi\|_{H^{m,s}}=\|\langle x\rangle^{s} 
\langle \nabla\rangle^{m} \phi\|_{L^2}<\infty\rbrace.
\end{eqnarray*}
 We denote 
various constants by $C$ and so forth. They may differ from 
line to line, when this does not cause any confusion. 

The plan of the present paper is as follows. 
In Section 2, we prove several linear estimates for the fourth order 
Schr\"{o}dinger type equation (\ref{4L}). In Section 3, we give 
several estimates for the asymptotic profile \ref{u}. 
In Section 4, we prove Theorems 
\ref{nonlinear1} and \ref{nonlinear2} by applying the contraction mapping principle to
the integral equation (\ref{MINT}). 
Finally in Section 5, we give several additional remarks.

\section{Linear Estimates} \label{sec:linear}

In this section, we derive several linear estimates for  
the fourth order Schr\"{o}dinger type equation
\begin{eqnarray}
\left\{
\begin{array}{l}
\displaystyle{
i\pt_tu+\frac12\Delta u-\frac{1}{4}\pt_{x_{1}}^4u=0,
\qquad t>0,x\in\rre^{d},
}\\
\displaystyle{u(0,x)=\psi(x),\qquad x\in\rre^{d}.}
\end{array}
\right.
\label{4LS}
\end{eqnarray}

\begin{proposition}\label{linear} 
Let $u$ be a solution to (\ref{4LS}). 
Then we have
\begin{eqnarray*}
u(t,x)=
\frac{t^{-\frac{d}{2}}}{\sqrt{3\mu_{1}^{2}+1}}\hat{\psi}(\mu)
e^{\frac34it\mu_{1}^{4}+\frac12it|\mu|^{2}
-i\frac{d}{4}\pi}+R(t,x)
\end{eqnarray*}
for $t\ge2$,
where $\mu=(\mu_{1},\mu_{\perp})$ is given by (\ref{mu}) 
and $R$ satisfies
\begin{eqnarray*}
\|\langle\pt_{x_{1}}\rangle^{1-\frac2p}R(t)\|_{L_x^{p}}
\le Ct^{-d(\frac12-\frac1p)-\beta}\|\psi\|_{H_{x}^{0,s}},
\end{eqnarray*}
for $2\le p\le\infty$, $3/8<\beta<1/2$ and $s>d/2-(d-1)/p+1$.
\end{proposition}
 
To calculate the oscillatory integral effectively, 
we show the following elementary lemma.

\begin{lemma}\label{lem1} Let 
$a,b,c\in\rre$ and let 
$\phi\in C^2([a,b],\rre)$ and $\psi\in C^1
([a,b],{{\mathbb C}})$. Then 
\begin{eqnarray*}
\lefteqn{\int_a^{b}e^{-i\phi(\xi)}
\psi(\xi)d\xi}
\\&=&
\left[e^{-i\phi(\xi)}\frac{(\xi-c)\psi(\xi)}{
1-i(\xi-c)\phi'(\xi)}\right]_{a}^{b}
-\int_{a}^{b}e^{-i\phi(\xi)}
\frac{(\xi-c)\psi'(\xi)}{1-i(\xi-c)\phi'(\xi)}d\xi
\\
& &-i\int_{a}^{b}e^{-i\phi(\xi)}
\frac{(\xi-c)\{\phi'(\xi)
+(\xi-c)\phi''(\xi)\}\psi(\xi)}{
\{1-i(\xi-c)\phi'(\xi)\}^2}d\xi.
\end{eqnarray*}
\end{lemma}

\begin{proof}[Proof of Lemma \ref{lem1}.] Lemma \ref{lem1}
follows from the combination of the 
identity 
\begin{eqnarray*}
e^{-i\phi(\xi)}=\frac{\{e^{-i\phi(\xi)}
(\xi-c)\}'}{
1-i(\xi-c)\phi'(\xi)}
\end{eqnarray*}
and  integration by parts. 
\end{proof}

\begin{proof}[Proof of Proposition \ref{linear}.]
Let $u$ be a solution to (\ref{4LS}). 
Then we have
\begin{eqnarray*}
u(t,x)
&=&
\left(\frac{1}{2\pi}\right)^{\frac{d}{2}}
\int_{\rre^{d}}e^{ix\xi-\frac{i}{2}t|\xi|^{2}-\frac{i}{4}t\xi_{1}^{4}}
\hat{\psi}(\xi)d\xi=\int_{\rre^{d}}K(t,x-y)\psi(y)dy,
\end{eqnarray*}
where
\begin{eqnarray*}
K(t,z)
=\left(\frac{1}{2\pi}\right)^{d}
\int_{\rre^{d}}e^{iz\xi-\frac{i}{2}t|\xi|^{2}-\frac{i}{4}t\xi_{1}^{4}}
d\xi.
\end{eqnarray*}
By the Fresnel integral formula
\begin{eqnarray*}
\frac{1}{\sqrt{2\pi}}
\int_{\rre}e^{iz_{j}\xi_{j}-\frac{i}{2}t\xi_{j}^{2}}d\xi_{j}
=t^{-\frac12}e^{\frac{iz_{j}^{2}}{2t}-i\frac{\pi}{4}},
\end{eqnarray*}
for $j=2.\cdots,d$, we have
\begin{eqnarray*}
K(t,z)
=\left(\frac{1}{2\pi}\right)^{\frac{d+1}{2}}
t^{-\frac{d-1}{2}}e^{\frac{i|z_{\perp}|^{2}}{2t}-i\frac{d-1}{4}\pi}
\int_{\rre}e^{iz_{1}\xi_{1}-\frac{i}{2}t\xi_{1}^{2}-\frac{i}{4}t\xi_{1}^{4}}
d\xi_{1}.
\end{eqnarray*}
Therefore, we find
\begin{eqnarray*}
u(t,x)
=\frac{1}{\sqrt{2\pi}}
t^{-\frac{d-1}{2}}e^{\frac{i}{2}t|\mu_{\perp}|^{2}-i\frac{d-1}{4}\pi}
\int_{\rre}e^{ix_{1}\xi_{1}-\frac{i}{2}t\xi_{1}^{2}-\frac{i}{4}t\xi_{1}^{4}}
{{\mathcal F}}[e^{\frac{i|y_{\perp}|^{2}}{2t}}\psi](\xi_{1},\frac{x_{\perp}}{t})d\xi_{1}.
\end{eqnarray*}
We split $u$ into the following two pieces: 
\begin{eqnarray}
\lefteqn{u(t,x)}
\nonumber\\
&=&\frac{1}{\sqrt{2\pi}}
t^{-\frac{d-1}{2}}e^{\frac{i}{2}t|\mu_{\perp}|^{2}-i\frac{d-1}{4}\pi}
\int_{\rre}e^{ix_{1}\xi_{1}-\frac{i}{2}t\xi_{1}^{2}-\frac{i}{4}t\xi_{1}^{4}}
{{\mathcal F}}[\psi](\xi_{1},\frac{x_{\perp}}{t})d\xi_{1}
\nonumber\\
& &+\frac{1}{\sqrt{2\pi}}
t^{-\frac{d-1}{2}}e^{\frac{i}{2}t|\mu_{\perp}|^{2}-i\frac{d-1}{4}\pi}
\nonumber\\
& &\qquad\times
\int_{\rre}e^{ix_{1}\xi_{1}-\frac{i}{2}t\xi_{1}^{2}-\frac{i}{4}t\xi_{1}^{4}}
{{\mathcal F}}[(e^{\frac{i|y_{\perp}|^{2}}{2t}}-1)\psi](\xi_{1},\frac{x_{\perp}}{t})d\xi_{1}
\nonumber\\
&=:&L+R.
\label{u1}
\end{eqnarray}
To evaluate $L$, we split $L$ into 
\begin{eqnarray}
L&=&
\frac{1}{\sqrt{2\pi}}
t^{-\frac{d-1}{2}}e^{\frac{i}{2}t|\mu_{\perp}|^{2}-i\frac{d-1}{4}\pi}
{{\mathcal F}}[\psi](\mu)
\int_{\rre}e^{ix_{1}\xi_{1}-\frac{i}{2}t\xi_{1}^{2}-\frac{i}{4}t\xi_{1}^{4}}
d\xi_{1}\nonumber\\
& &+\frac{1}{\sqrt{2\pi}}
t^{-\frac{d-1}{2}}e^{\frac{i}{2}t|\mu_{\perp}|^{2}-i\frac{d-1}{4}\pi}
\nonumber\\
& &\qquad\times
\int_{\rre}e^{ix_{1}\xi_{1}-\frac{i}{2}t\xi_{1}^{2}-\frac{i}{4}t\xi_{1}^{4}}
({{\mathcal F}}[\psi](\xi_{1},\frac{x_{\perp}}{t})
-{{\mathcal F}}[\psi](\mu_{1},\frac{x_{\perp}}{t}))d\xi_{1}\nonumber\\
&=:&L_{1}(t,x)+L_{2}(t,x).
\label{u2}
\end{eqnarray}
We rewrite $L_{1}$ as follows:
\begin{eqnarray*}
L_{1}(t,x)=
\frac{1}{\sqrt{2\pi}}
t^{-\frac{d-1}{2}}
e^{\frac34it\mu_{1}^{4}+\frac{i}{2}t|\mu|^{2}-i\frac{d-1}{4}\pi}
{{\mathcal F}}[\psi](\mu)
\int_{\rre}e^{-itS(\mu_{1},\xi_{1})}d\xi_{1},
\end{eqnarray*}
where 
$S(\mu_{1},\xi_{1})$ is defined by 
\begin{eqnarray*}
S(\mu_{1},\xi_{1})
&=&\frac{1}{4}\xi_{1}^{4}+\frac{1}{2}\xi_{1}^{2}
-(\mu_{1}^{3}+\mu)\xi_{1}+\frac34\mu_{1}^{4}+\frac12\mu_{1}^{2}.
\end{eqnarray*}
Let
\begin{eqnarray*}
\eta_{1}
=\mu_{1}+\frac{1}{\sqrt{2}}\frac{1}{\sqrt{3\mu_{1}^2+1}}
(\xi_{1}-\mu_{1})
\sqrt{\xi_{1}^2+2\mu_{1}\xi_{1}+3\mu_{1}^2+2}.
\end{eqnarray*}
We rewrite $L_{1}$ as follows:
\begin{eqnarray}
\lefteqn{L_{1}(t,x)}\nonumber\\
&=&
\frac{1}{\sqrt{2\pi}}
t^{-\frac{d-1}{2}}
{{\mathcal F}}[\psi](\mu)
e^{\frac34it\mu_{1}^{4}+\frac{i}{2}t|\mu|^{2}-i\frac{d-1}{4}\pi}
\nonumber\\
& &\times
\left(\int_{\rre}
e^{-itS(\mu_{1},\xi_{1})}
\frac{d\eta_{1}}{d\xi_{1}}d\xi_{1}
+
\int_{\rre}
e^{-itS(\mu_{1}\xi_{1})}
(1-\frac{d\eta_{1}}{d\xi_{1}})d\xi_{1}
\right)\nonumber\\
&=:&
t^{-\frac{d-1}{2}}
{{\mathcal F}}[\psi](\mu)
e^{\frac34it\mu_{1}^{4}+\frac{i}{2}t|\mu|^{2}-i\frac{d-1}{4}\pi}
L_{1,1}(t,x)
\nonumber\\
& &+t^{-\frac{d-1}{2}}{{\mathcal F}}[\psi](\mu)e^{\frac{i}{2}t|\mu_{\perp}|^{2}-i\frac{d-1}{4}\pi}
L_{1,2}(t,x).
\label{u3}
\end{eqnarray}
For $L_{1,1}$, changing the variable $\xi_{1}\mapsto\eta_{1}$, 
we have
\begin{eqnarray*}
L_{1,1}(t,x)=\frac{1}{\sqrt{2\pi}}
\int_{\rre}
e^{-\frac12it(3\mu_{1}^2+1)(\eta_{1}-\mu_{1})^2}d\eta_{1}.
\end{eqnarray*}
In addition, changing the variable $\zeta_{1}=(1/\sqrt{2})t^{1/2}
\sqrt{3\mu_{1}^2+1}(\eta_{1}-\mu_{1})$ ($\eta_{1}\mapsto\zeta_{1}$) and 
using the Fresnel integral formula, we obtain
\begin{eqnarray}
L_{1,1}(t,x)=\sqrt{\frac{2}{\pi}}
\frac{t^{-1/2}}{\sqrt{3\mu_{1}^{2}+1}}
\int_{\rre}e^{-i\zeta^2}d\zeta
=
\frac{t^{-1/2}}{\sqrt{3\mu_{1}^{2}+1}}e^{-i\frac{\pi}{4}}.
\label{u4}
\end{eqnarray}
For $L_{1,2}$, using Lemma \ref{lem1}, we have
\begin{eqnarray}
\lefteqn{L_{1,2}(t,x)}
\nonumber\\
&=&\frac{1}{\sqrt{2\pi}}\int_{\rre}e^{ix_{1}\xi_{1}-\frac{i}{2}t\xi_{1}^{2}-\frac{i}{4}t\xi_{1}^{4}}
(1-\frac{d\eta_{1}}{d\xi_{1}})d\xi_{1}\nonumber\\
&=&
\frac{1}{\sqrt{2\pi}}\int_{\rre}e^{ix_{1}\xi_{1}-\frac{i}{2}t\xi_{1}^{2}-\frac{i}{4}t\xi_{1}^{4}}
\frac{(\xi_{1}-\mu_{1})}{1-it(\xi_{1}-\mu_{1})^{2}
(\xi_{1}^{2}+\mu_{1}\xi_{1}+\mu_{1}^{2}+1)}\frac{d^2\eta_{1}}{d\xi_{1}^2}d\xi_{1}
\nonumber\\
& &-\frac{it}{\sqrt{2\pi}}
\int_{\rre}e^{ix_{1}\xi_{1}-\frac{i}{2}t\xi_{1}^{2}-\frac{i}{4}t\xi_{1}^{4}}
\frac{(\xi_{1}-\mu_{1})^2(4\xi_{1}^{2}+\mu_{1}\xi_{1}+\mu_{1}^{2}+2)}{
\{1-it(\xi_1-\mu_{1})^{2}(\xi_{1}^{2}+\mu_{1}\xi_{1}+\mu_{1}^{2}+1)\}^2}
\nonumber\\
& &\qquad\qquad\qquad\qquad\qquad\times
(1-\frac{d\eta_{1}}{d\xi_{1}})d\xi_{1}.\nonumber
\end{eqnarray}
Since
\begin{eqnarray*}
\frac{d\eta_{1}}{d\xi_{1}}&=&
\sqrt{2}\frac{1}{\sqrt{3\mu_{1}^2+1}}
\frac{\xi_{1}^2+\mu_{1}\xi_{1}
+\mu_{1}^{2}+1}
{\sqrt{\xi_{1}^2+2\mu_{1}\xi_{1}+3\mu_{1}^2+2}},\\
\frac{d^2\eta_{1}}{d\xi_{1}^2}&=&
\sqrt{2}\frac{1}{\sqrt{3\mu_{1}^2+1}}
\frac{\xi_{1}^3+3\mu_{1}\xi_{1}^2
+(6\mu_{1}^2+3)\xi_{1}
+(2\mu_{1}^3+\mu_{1})}
{(\xi_{1}^2+2\mu_{1}\xi_{1}+3\mu_{1}^2+2)^{3/2}},
\end{eqnarray*}
we see that
\begin{eqnarray*}
\sup_{\xi_{1}\in\rre}\left|\frac{d^2\eta_{1}}{d\xi_{1}^2}\right|
&\le&\frac{C}{\sqrt{3\mu_{1}^2+1}},\\
\left|1-\frac{d\eta_{1}}{d\xi_{1}}\right|&=&
\left|\frac{d\eta_{1}}{d\xi_{1}}\biggl|_{\xi_{1}=\mu_{1}}
-\frac{d\eta_{1}}{d\xi_{1}}\right|
\le\sup_{\xi_{1}\in\rre}\left|\frac{d^2\eta_{1}}{d\xi_{1}^2}\right|
|\xi_{1}-\mu_{1}|
\le C\frac{|\xi_{1}-\mu_{1}|}{\sqrt{3\mu_{1}^2+1}}.
\end{eqnarray*}
Furthermore, we easily see that 
\begin{eqnarray}
\sup_{\xi_{1}\in\rre}\left|t
\frac{(\xi_{1}-\mu_{1})^2(4\xi_{1}^{2}+\mu_{1}\xi_{1}+\mu_{1}^{2}+2)}{
1-it(\xi_{1}-\mu_{1})^{2}(\xi_{1}^{2}+\mu_{1}\xi_{1}+\mu_{1}^{2}+1)}
\right|\le C.\label{se}
\end{eqnarray}
Combining above three  
inequalities, we have
\begin{eqnarray*}
|\langle\pt_{x_{1}}\rangle^{1-\frac2p}L_{1,2}(t,x)|
\le \frac{C}{\sqrt{3\mu_{1}^2+1}}
\int_{\rre}
\frac{\langle\xi_{1}\rangle^{1-\frac2p}|\xi_{1}-\mu_{1}|}{
1+t(\xi_{1}-\mu_{1})^2(\xi_{1}^2+\mu_{1}\xi_{1}+\mu_{1}^2+1)}d\xi_{1}.
\end{eqnarray*}
We split the right hand side of above inequality into 
the following two pieces:
\begin{eqnarray*}
\lefteqn{|\langle\pt_{x_{1}}\rangle^{1-\frac2p}L_{1,2}(t,x)|}\\
&\le&
\frac{C}{\sqrt{3\mu_{1}^2+1}}
\int_{|\xi_{1}-\mu_{1}|\le\sqrt{\mu_{1}^2+1}}
\frac{\langle\xi_{1}\rangle^{1-\frac2p}|\xi_{1}-\mu_{1}|}{
1+t(\xi_{1}-\mu_{1})^2(\xi_{1}^2+\mu_{1}\xi_{1}+\mu_{1}^2+1)}d\xi_{1}\\
& &
+\frac{C}{\sqrt{3\mu_{1}^2+1}}\int_{|\xi_{1}-\mu_{1}|\ge\sqrt{\mu_{1}^2+1}}
\frac{\langle\xi_{1}\rangle^{1-\frac2p}|\xi_{1}-\mu_{1}|}{
1+t(\xi_{1}-\mu_{1})^2(\xi_{1}^2+\mu_{1}\xi_{1}+\mu_{1}^2+1)}d\xi_{1}.
\end{eqnarray*}
Using the inequalities 
\begin{eqnarray*}
\xi_{1}^2+\mu_{1}\xi_{1}+\mu_{1}^2+1\ge
\left\{
\begin{array}{l}
\medskip\displaystyle{
\frac12(\mu_{1}^2+1),\ \qquad\text{if}
\ |\xi_{1}-\mu_{1}|\le\sqrt{\mu_{1}^2+1},}\\
\displaystyle{
\frac14(\xi-\mu_{1})^2,\qquad\text{if}
\ |\xi_{1}-\mu_{1}|\ge\sqrt{\mu_{1}^2+1},}
\end{array}
\right.
\end{eqnarray*}
we have 
\begin{eqnarray}
\lefteqn{|\langle\pt_{x_{1}}\rangle^{1-\frac2p}L_{1,2}(t,x)|}\nonumber\\
&\le&\frac{C}{\sqrt{3\mu_{1}^2+1}}
\int_{|\xi_{1}-\mu_{1}|\le\sqrt{\mu_{1}^2+1}}
\frac{\langle\xi_{1}\rangle^{1-\frac2p}|\xi_{1}-\mu_{1}|}{
1+t(\mu_{1}^2+1)(\xi_{1}-\mu_{1})^2}d\xi_{1}
\nonumber\\
& &
+\frac{C}{\sqrt{3\mu_{1}^2+1}}
\int_{|\xi_{1}-\mu_{1}|\ge\sqrt{\mu_{1}^2+1}}
\frac{\langle\xi_{1}\rangle^{1-\frac2p}|\xi_{1}-\mu_{1}|}{1+t|\xi_{1}-\mu_{1}|^4}d\xi_{1}\nonumber\\
&\le&C
t^{-\beta-\frac12}(\mu_{1}^2+1)^{-\beta-\frac1p-\frac12}
\int_{|\xi_{1}-\mu_{1}|\le\sqrt{\mu_{1}^2+1}}
|\xi_{1}-\mu_{1}|^{-2\beta}d\xi_{1}
\nonumber\\
& &
+C
t^{-\beta-\frac12}(\mu_{1}^2+1)^{-\frac12}
\int_{|\xi_{1}-\mu_{1}|\ge\sqrt{\mu_{1}^2+1}}
|\xi_{1}-\mu_{1}|^{-4\beta-\frac2p}d\xi_{1}
\nonumber\\
&\le&Ct^{-\beta-\frac12}(\mu_{1}^2+1)^{-2\beta-\frac1p},
\label{u5}
\end{eqnarray}
where $\max(0,1/4-1/(2p))<\beta<1/2$. 
By (\ref{u3}), (\ref{u4}) and (\ref{u5}), we have
\begin{eqnarray}
L_{1}(t,x)
=\frac{t^{-\frac{d}{2}}}{\sqrt{3\mu_{1}^{2}+1}}{{\mathcal F}}[\psi](\mu)
e^{\frac34it\mu_{1}^{4}+\frac{i}{2}t|\mu|^{2}-i\frac{d}{4}\pi}
+R_{1}(t,x),\label{l}
\end{eqnarray}
where $R_{1}$ satisfies
\begin{eqnarray*}
|\langle\pt_{x_{1}}\rangle^{1-\frac2p}R_{1}(t,x)|\le 
Ct^{-\frac{d}{2}-\beta}(\mu_{1}^2+1)^{-2\beta-\frac1p}
|{{\mathcal F}}[\psi](\mu)|.
\end{eqnarray*}
Hence the H\"older and Sobolev inequalities yield 
\begin{eqnarray*}
\|\langle\pt_{x_{1}}\rangle^{1-\frac2p}
R_{1}(t)\|_{L_{x_{1}}^{p}}
&\le& Ct^{-\frac{d}{2}-\beta}
\|(\mu_{1}^2+1)^{-2\beta-\frac1p}
{{\mathcal F}}[\psi](\mu)
\|_{L_{x_{1}}^{p}}\\
&\le& Ct^{-\frac{d}{2}-\beta}
\|(\mu_{1}^2+1)^{-2\beta-\frac1p}\|_{L_{x_{1}}^{p}}
\|{{\mathcal F}}[\psi](\cdot,\mu_{\perp})
\|_{L_{\xi_{1}}^{\infty}}\\
&\le& Ct^{\frac1p-\frac{d}{2}-\beta}
\|{{\mathcal F}}[\psi](\cdot,\mu_{\perp})
\|_{L_{\xi_{1}}^{\infty}}\\
&\le& Ct^{\frac1p-\frac{d}{2}-\beta}
\|{{\mathcal F}}[\psi](\cdot,\mu_{\perp})
\|_{H_{\xi_{1}}^{s_{1}}}
\end{eqnarray*}
for $2\le p\le\infty$ and $1/(4p)<\beta<1/2$, where $s_{1}>1/2$. 
Combining the above inequality and the Minkowski inequality, 
we have
\begin{eqnarray}
\|\langle\pt_{x_{1}}\rangle^{1-\frac2p}
R_{1}(t)\|_{L_{x}^{p}}
&\le&Ct^{\frac1p-\frac{d}{2}-\beta}
\|\|{{\mathcal F}}[\psi](\xi_{1},\mu_{\perp})
\|_{H_{\xi_{1}}^{s_{1}}}\|_{L_{x_{\perp}}^{p}}
\nonumber\\
&\le&Ct^{\frac1p-\frac{d}{2}-\beta}
\|\|{{\mathcal F}}[\psi](\xi_{1},\mu_{\perp})
\|_{L_{x_{\perp}}^{p}}\|_{H_{\xi_{1}}^{s_{1}}}
\nonumber\\
&\le&Ct^{-d(\frac12-\frac1p)-\beta}
\|\|{{\mathcal F}}[\psi]\|_{L_{\xi_{\perp}}^{p}}
\|_{H_{\xi_{1}}^{s_{1}}}
\nonumber\\
&\le&Ct^{-d(\frac12-\frac1p)-\beta}
\|\|{{\mathcal F}}[\psi]\|_{H_{\xi_{\perp}}^{s_{\perp}}}
\|_{H_{\xi_{1}}^{s_{1}}}
\nonumber\\
&\le&Ct^{-d(\frac12-\frac1p)-\beta}
\|\psi\|_{H^{0,s}},\label{r11}
\end{eqnarray}
where $s=s_{1}+s_{\perp}$ and 
$s_{\perp}=(d-1)(1/2-1/p)$ for $2\le p<\infty$ 
and $s_{\perp}>(d-1)/2$ for $p=\infty$.

Next we evaluate $L_{2}$. We write 
\begin{eqnarray*}
L_{2}(t,x)=:
t^{-\frac{d-1}{2}}e^{\frac{i|x_{\perp}|^{2}}{2t}-i\frac{d-1}{4}\pi}\tilde{L}_{2}(t,x).
\end{eqnarray*}
Using Lemma \ref{lem1}, we have
\begin{eqnarray}
\lefteqn{
\langle\pt_{x_{1}}\rangle^{1-\frac2p}
\tilde{L}_{2}(t,x)}\nonumber\\
&=&
-\frac{1}{\sqrt{2\pi}}\int_{\rre}e^{ix_{1}\xi_{1}-\frac{i}{2}t\xi_{1}^{2}-\frac{i}{4}t\xi_{1}^{4}}
\frac{\langle\xi_{1}\rangle^{1-\frac2p}(\xi_{1}-\mu_{1})}{1-it(\xi_{1}-\mu_{1})^{2}
(\xi_{1}^{2}+\mu_{1}\xi_{1}+\mu_{1}^{2}+1)}
\nonumber\\
& &\qquad\qquad\qquad\times
\pt_{\xi_{1}}{{\mathcal F}}[\psi](\xi_{1},\frac{x_{\perp}}{t})d\xi_{1}
\nonumber\\
& &-
\frac{it}{\sqrt{2\pi}}\int_{\rre}e^{ix_{1}\xi_{1}-\frac{i}{2}t\xi_{1}^{2}-\frac{i}{4}t\xi_{1}^{4}}
\frac{\langle\xi_{1}\rangle^{1-\frac2p}
(\xi_{1}-\mu_{1})^{2}(4\xi_{1}^{2}+\mu_{1}\xi_{1}+\mu_{1}^{2}+2)}{
\{1-it(\xi_{1}-\mu_{1})^{2}(\xi_{1}^{2}+\mu_{1}\xi_{1}+\mu_{1}^{2}+1)\}^2}
\nonumber\\
& &\qquad\qquad\qquad\times
({{\mathcal F}}[\psi](\xi_{1},\frac{x_{\perp}}{t})
-{{\mathcal F}}[\psi](\mu_{1},\frac{x_{\perp}}{t}))d\xi_{1}.\nonumber
\end{eqnarray}
Form (\ref{se}), we have
\begin{eqnarray*}
\lefteqn{|\langle\pt_{x_{1}}\rangle^{1-\frac2p}\tilde{L}_{2}(t,x)|}\\
&\le& 
C\|\pt_{\xi_{1}}{{\mathcal F}}[\psi]
(\cdot,\frac{x_{\perp}}{t})\|_{L_{\xi_{1}}^{\infty}}
\int_{\rre}
\frac{\langle\xi_{1}\rangle^{1-\frac2p}|\xi_{1}-\mu_{1}|}{
1+t(\xi_{1}-\mu_{1})^{2}(\xi_{1}^2+\mu_{1}\xi_{1}+\mu_{1}^2+1)}d\xi_{1}.
\end{eqnarray*}
The same argument as that in (\ref{u5}) yields 
\begin{eqnarray*}
|\langle\pt_{x_{1}}\rangle^{1-\frac2p}\tilde{L}_{2}(t,x)|
\le Ct^{-\beta-1/2}(\mu_{1}^{2}+1)^{-2\beta-\frac1p+\frac12}\|\pt_{\xi_{1}}{{\mathcal F}}[\psi]
(\cdot,\frac{x_{\perp}}{t})\|_{L_{\xi_{1}}^{\infty}}.
\end{eqnarray*}
Hence
\begin{eqnarray*}
\lefteqn{\|\langle\pt_{x_{1}}\rangle^{1-\frac2p} L_{2}(t)\|_{L_{x_{1}}^{p}}}\\
&\le& Ct^{-\frac{d}{2}-\beta}
\|(\mu_{1}^2+1)^{-2\beta-\frac1p+\frac12}\|_{L_{x_{1}}^{p}}
\|\pt_{\xi_{1}}{{\mathcal F}}[\psi](\cdot,\frac{x_{\perp}}{t})
\|_{L_{\xi_{1}}^{\infty}}\\
&\le& Ct^{\frac1p-\frac{d}{2}-\beta}
\|(\mu_{1}^2+1)^{-2\beta+\frac12}\|_{L_{\mu_{1}}^{p}}
\|\pt_{\xi_{1}}{{\mathcal F}}[\psi](\cdot,\frac{x_{\perp}}{t})
\|_{L_{\xi_{1}}^{\infty}}\\
&\le& Ct^{\frac1p-\frac{d}{2}-\beta}
\|\pt_{\xi_{1}}{{\mathcal F}}[\psi](\cdot,\frac{x_{\perp}}{t})
\|_{H_{\xi_{1}}^{s_{1}}}
\end{eqnarray*}
for $2\le p\le\infty$ and $1/4+1/(4p)<\beta<1/2$. 
Combining the above inequality and the Minkowski inequality, 
we have 
\begin{eqnarray}
\|\langle\pt_{x_{1}}\rangle^{1-\frac2p}L_{2}(t)\|_{L_{x}^{p}}
&\le&Ct^{\frac1p-\frac{d}{2}-\beta}
\|\|\pt_{\xi_{1}}{{\mathcal F}}[\psi](\cdot,\frac{x_{\perp}}{t})
\|_{H_{\xi_{1}}^{s_{1}}}\|_{L_{x_{\perp}}^{p}}
\nonumber\\
&\le&Ct^{\frac1p-\frac{d}{2}-\beta}
\|\|\pt_{\xi_{1}}{{\mathcal F}}[\psi](\xi_{1},\frac{\cdot}{t})
\|_{L_{x_{\perp}}^{p}}\|_{H_{\xi_{1}}^{s_{1}}}
\nonumber\\
&\le&Ct^{-d(\frac12-\frac1p)-\beta}
\|\|\pt_{\xi_{1}}{{\mathcal F}}[\psi]\|_{L_{\xi_{\perp}}^{p}}
\|_{H_{\xi_{1}}^{s_{1}}}
\nonumber\\
&\le&Ct^{-d(\frac12-\frac1p)-\beta}
\|\|\pt_{\xi_{1}}{{\mathcal F}}[\psi]\|_{H_{\xi_{\perp}}^{s_{\perp}}}
\|_{H_{\xi_{1}}^{s_{1}}}
\nonumber\\
&\le&Ct^{-d(\frac12-\frac1p)-\beta}
\|\psi\|_{H^{0,s+1}}.\label{r12}
\end{eqnarray}
Finally let us evaluate $R$. $R$ can be rewritten as 
\begin{eqnarray*}
R=
t^{-\frac{d-1}{2}}e^{\frac{i}{2}t|\mu_{\perp}|^{2}-i\frac{d-1}{4}\pi}
W_{4LS}(t){{\mathcal F}}_{\perp}
[(e^{\frac{i|y_{\perp}|^{2}}{2t}}-1)\psi](x_{1},\frac{x_{\perp}}{t}),
\end{eqnarray*}
where $\{W_{4LS}(t)\}_{t\in\rre}$ is a unitary group generated by 
the linear operator $(i/2)\pt_{x_{1}}^{2}
-(i/4)\pt_{x_{1}}^{4}$:
\begin{eqnarray*}
W_{4LS}(t)\phi=\frac{1}{\sqrt{2\pi}}
\int_{\rre}e^{ix_{1}\xi_{1}-\frac{i}{2}t\xi_{1}^{2}-\frac{i}{4}t\xi_{1}^{4}}
\hat{\phi}(\xi_{1})d\xi_{1}.
\end{eqnarray*}
By using the decay estimate (see \cite{BKS,S} for instance), 
\begin{eqnarray*}
\|\langle\pt_{x_{1}}\rangle^{1-\frac2p}W_{4LS}(t)\psi\|_{L_{x_{1}}^{p}}
\le Ct^{-(\frac12-\frac1p)}\|\psi\|_{L_{x_{1}}^{p'}},
\end{eqnarray*}
we obtain
\begin{eqnarray*}
\|\langle\pt_{x_{1}}\rangle^{1-\frac2p}R(t)\|_{L_{x_{1}}^{p}}
&\le&Ct^{-\frac{d}{2}+\frac1p}\|
{{\mathcal F}}_{\perp}[(e^{\frac{i|y_{\perp}|^{2}}{2t}}-1)
\psi](x_{1},\frac{x_{\perp}}{t})\|_{L_{x_{1}}^{p'}}.
\end{eqnarray*}
Combining the above inequality and the Minkowski inequality, 
we have
\begin{eqnarray}
\lefteqn{\|\langle\pt_{x_{1}}\rangle^{1-\frac2p}R(t)\|_{L_{x}^{p}}}
\nonumber\\
&\le&Ct^{-\frac{d}{2}+\frac1p}\|\|
{{\mathcal F}}_{\perp}[(e^{\frac{i|y_{\perp}|^{2}}{2t}}-1)
\psi](x_{1},\frac{x_{\perp}}{t})\|_{L_{x_{1}}^{p'}}\|_{L_{x_{\perp}}^{p}}
\nonumber\\
&\le&Ct^{-\frac{d}{2}+\frac1p}\|\|
{{\mathcal F}}_{\perp}[(e^{\frac{i|y_{\perp}|^{2}}{2t}}-1)
\psi](x_{1},\frac{x_{\perp}}{t})\|_{L_{x_{\perp}}^{p}}\|_{L_{x_{1}}^{p'}}
\nonumber\\
&\le&Ct^{-d(\frac12-\frac1p)}\|
(e^{\frac{i|x_{\perp}|^{2}}{2t}}-1)
\psi\|_{L_{x}^{p'}}\nonumber\\
&\le&Ct^{-d(\frac12-\frac1p)-\beta}\|
|x_{\perp}|^{2\beta}\psi\|_{L_{x}^{p'}}
\nonumber\\
&\le&Ct^{-d(\frac12-\frac1p)-\beta}\|
\psi\|_{H_{x}^{0,s}},
\label{r14}
\end{eqnarray}
where $0<\beta<1$ and $s>d/2-d/p+2\beta$. Collecting (\ref{u1}), (\ref{u2}), (\ref{l}), 
(\ref{r11}),  (\ref{r12}) and (\ref{r14}), 
we obtain the desired result. 
\end{proof}

To prove Theorems \ref{nonlinear1} and \ref{nonlinear2}, 
we employ the Strichartz 
estimate for 
the linear fourth order Schr\"{o}dinger equation 
(\ref{4LS}). 

\begin{lemma}\label{S} 
Let $d\ge3$ and let $(q_{j},r_{j})$ ($j=1,2$) satisfy
\[
	\frac{2}{q_{j}}+\frac{d}{r_{j}}=\frac{d}{2},\qquad 2\le r_{j}\le\frac{2d}{d-2}\ 
	\text{and}\ (q_{j},r_{j},d)\neq(2,\infty,2).
\]
Then, the inequalitiy 
\begin{equation*}
\left\|\langle\pt_{x_{1}}\rangle^{\frac{2}{q_{1}d}}
\int_t^{+\infty}W(t-t')F(t')dt'
\right\|_{L_t^{q_{1}}(t,\infty;L_{x}^{r_{1}})}
\le C
\|\langle\pt_{x_{1}}\rangle^{-\frac{2}{q_{2}d}}
F\|_{L_t^{q_{2}'}(t,\infty;L_{x}^{r_{2}'})} 
\end{equation*}
holds.
\end{lemma}

\begin{proof}[Proof of Lemma \ref{S}.] For the non-endpoint case 
(i.e., $2\le r_{j}\le2d/(d-2)$ for any $j=1,2$), 
see Kenig, Ponce and Vega \cite[Theorem 3.1]{KPV}. 
For the endpoint case 
(i.e., $2\le r_{j}\le2d/(d-2)$ for some $j=1,2$), 
see Keel and Tao \cite[Theorem 10.2]{KT}. 
\end{proof}

\section{Nonlinear Estimates} \label{sec:linear}

In this section, we derive several estimates 
for the asymptotic profile \eqref{u}. 

\begin{proposition}\label{nl1} 
Assume $d=2,3$
and let $S_{+}$ be given by (\ref{u}). Then 
if $1<s<2$ and $s<p<3$, then we have for $t\ge1$,
\begin{eqnarray*}
\|\hat{\psi}_{+}e^{iS_{+}(t,\xi)}\|_{H^{s}}
&\le& C
(1+\|\hat{\psi}_{+}\|_{L^{\infty}}^{2p-2})\|
\psi_{+}\|_{H^{0,s}},\\
\lefteqn{\|\langle\sqrt{3}\xi_{1}\rangle^{-p+1}
|\hat{\psi}_{+}|^{p-1}\hat{\psi}_{+}e^{iS_{+}(t,\xi)}\|_{H^{s}}}
\qquad\qquad\qquad\ \ \\
&\le&C
(1+\|\hat{\psi}_{+}\|_{L^{\infty}}^{2p-2})
\|\hat{\psi}_{+}\|_{L^{\infty}}^{p-1}
\|\psi_{+}\|_{H^{0,s}}.
\end{eqnarray*}
If $2\le s<p<3$, then we have for $t\ge1$,
\begin{eqnarray*}
\|\hat{\psi}_{+}e^{iS_{+}(t,\xi)}\|_{H^{s}}
&\le&
C(1+\|\hat{\psi}_{+}\|_{L^{\infty}}^{3p-3})
\|{\psi}_{+}\|_{H^{0,s}},\\
\lefteqn{\|\langle\sqrt{3}\xi_{1}\rangle^{-p+1}
|\hat{\psi}_{+}|^{p-1}\hat{\psi}_{+}e^{iS_{+}(t,\xi)}\|_{H^{s}}}
\qquad\qquad\qquad\ \ \\
&\le&
C(1+\|\hat{\psi}_{+}\|_{L^{\infty}}^{3p-3})
\|\hat{\psi}_{+}\|_{L^{\infty}}^{p-1}
\|\psi_{+}\|_{H^{0,s}}.
\end{eqnarray*}
In particular, if we further assume $s>d/2$, then we have
\begin{eqnarray*}
\|\hat{\psi}_{+}e^{iS_{+}(t,\xi)}\|_{H^{s}}
&\le& P(\|\psi_{+}\|_{H^{0,s}}),\\
\|\langle\sqrt{3}\xi_{1}\rangle^{-\frac{p-1}{2}}
|\hat{\psi}_{+}|^{p-1}\hat{\psi}_{+}e^{iS_{+}(t,\xi)}\|_{H^{s}}
&\le& P(\|\psi_{+}\|_{H^{0,s}}),
\end{eqnarray*}
where $P$ is a polynomial without constant term. 
\end{proposition}

To prove Proposition \ref{nl1}, we employ the 
Leibniz and chain rules for the fractional derivatives. 

\begin{lemma}[Fractional Leibniz rule]\label{pp} 
Assume that $s \ge 0$. 
Let $1<p,p_1,p_{4}<\infty$ and $1<p_{2},p_{3}\le\infty$.
Then, we have
\begin{eqnarray*}
\||\nabla|^{s}(fg)\|_{L_{x}^{p}} 
\le C(\||\nabla|^{s}f\|_{L_{x}^{p_{1}}}
\|g\|_{L_{x}^{p_{2}}}
+ \|f\|_{L_{x}^{p_{3}}}
\||\nabla|^{s} g\|_{L_{x}^{p_{4}}}),
\end{eqnarray*}
provided that 
$1/p=1/p_{1}+1/p_{2}=1/p_{3}+1/p_{4}$.
\end{lemma}

\begin{proof}[Proof of Lemma \ref{pp}.] 
See \cite[Proposition 3.3]{CW} for instance. 
\end{proof}

\begin{lemma}[Fractional chain rules]\label{ff} 

\vskip1mm
\noindent
(i) Assume $G\in C^{1}({\mathbb C},{\mathbb C})$ 
and $0<s\le1$. Let $1<p,p_{2}<\infty$ and 
$1<p_{1}\le\infty$. 
Then we have 
\begin{eqnarray*}
\||\nabla|^{s}G(u)\|_{L^{p}}
\le C\|G'(u)\|_{L^{p_{1}}}\||\nabla|^{s}u\|_{L^{p_{2}}},
\end{eqnarray*}
provided $1/p=1/p_{1}+1/p_{2}$.

\vskip1mm
\noindent
(ii) Assume $G\in C^{0,\alpha}({\mathbb C},{\mathbb C})$ 
for some $0<\alpha<1$. Then for any $0<s<\alpha, 
1<p<\infty$ and $s/\alpha<\sigma<1$, we have
\begin{eqnarray*}
\||\nabla|^{s}G(u)\|_{L^{p}}
\le C\||u|^{\alpha-\frac{s}{\sigma}}\|_{L^{p_{1}}}
\||\nabla|^{\sigma}u\|_{L^{\frac{s}{\sigma}p_{2}}}^{\frac{s}{\sigma}},
\end{eqnarray*}
provided that $1/p=1/p_{1}+1/p_{2}$ and $(1-s/(\alpha\sigma))p_{1}>1$

\end{lemma}

\begin{proof}[Proof of Lemma \ref{ff}.] 
For the proofs for (i) and (ii), see \cite[Proposition 3.1]{CW} 
and \cite[Proposition A.1]{V}, respectively. 
\end{proof}

\begin{lemma}\label{nl}
Let $f(u)=|u|^{q-1}u$, $q=1$ or $p$ and $S(u)=\mu|u|^{p-1}t^{-\delta}$. 
If $1<s<2$ and $s<p<3$, then we have for $t\ge1$,
\begin{eqnarray*}
\lefteqn{\|\langle\sqrt{3}\xi_{1}\rangle f(u)e^{iS(u)}\|_{H^{s}}}\\
&\le& C\left\{
\begin{array}{l}
\medskip\displaystyle{
(1+\|\langle\sqrt{3}\xi_{1}\rangle 
u\|_{L^{\infty}}^{2p-2})\|\langle\sqrt{3}\xi_{1}\rangle
u\|_{H^{s}}
\qquad \text{for}\ q=1,}\\
\displaystyle{
(1+\|\langle\sqrt{3}\xi_{1}\rangle u\|_{L^{\infty}}^{2p-2})
\|\langle\sqrt{3}\xi_{1}\rangle u\|_{L^{\infty}}^{p-1}
\|\langle\sqrt{3}\xi_{1}\rangle u\|_{H^{s}}
\qquad \text{for}\ q=p.}
\end{array}
\right.
\end{eqnarray*}
If $2\le s<p<3$, then we have for $t\ge1$,
\begin{eqnarray*}
\lefteqn{\|\langle\sqrt{3}\xi_{1}\rangle f(u)e^{iS(u)}\|_{H^{s}}}\\
&\le& C\left\{
\begin{array}{l}
\medskip\displaystyle{
(1+\|\langle\sqrt{3}\xi_{1}\rangle u\|_{L^{\infty}}^{3p-3})
\|\langle\sqrt{3}\xi_{1}\rangle u\|_{H^{s}}
\qquad \text{for}\ q=1,}\\
\displaystyle{
(1+\|\langle\sqrt{3}\xi_{1}\rangle u\|_{L^{\infty}}^{3p-3})
\|\langle\sqrt{3}\xi_{1}\rangle u\|_{L^{\infty}}^{p-1}
\|\langle\sqrt{3}\xi_{1}\rangle u\|_{H^{s}}
\qquad \text{for}\ q=p.}
\end{array}
\right.
\end{eqnarray*}
\end{lemma}

\begin{proof}[Proof of Lemma \ref{nl}.] 
We consider the case $q=1$ only since the case 
$q=p$ being similar. 
An $L^{2}$ estimate for $\langle\sqrt{3}\xi_{1}\rangle f(u)e^{iS(u)}$ 
is trivial, so we consider $\dot{H}^{s}$ estimate. 

We first consider the case $1<s<p<2$. Since 
\begin{eqnarray*}
\nabla(\langle\sqrt{3}\xi_{1}\rangle ue^{iS(u)})
&=&\nabla(\langle\sqrt{3}\xi_{1}\rangle u)(1+iS'(u)u)e^{iS(u)}\\
& &-iS'(u)u([\nabla,\langle\sqrt{3}\xi_{1}\rangle]u)e^{iS(u)},
\end{eqnarray*}
we obtain
\begin{eqnarray*}
\lefteqn{\||\nabla|^{s}(\langle\sqrt{3}\xi_{1}\rangle ue^{iS(u)})\|_{L^{2}}}\\
&\simeq&\||\nabla|^{s-1}\nabla(\langle\sqrt{3}\xi_{1}\rangle ue^{iS(u)})\|_{L^{2}}\\
&\le&
C(\||\nabla|^{s-1}\{\nabla(\langle\sqrt{3}\xi_{1}\rangle u)
e^{iS(u)}\}\|_{L^{2}}
+\||\nabla|^{s-1}\{\nabla(\langle\sqrt{3}\xi_{1}\rangle u)
S'(u)u e^{iS(u)}\}\|_{L^{2}})\\
& &+C\||\nabla|^{s-1}\{
([\nabla,\langle\sqrt{3}\xi_{1}\rangle]u)S'(u)ue^{iS(u)}\}\|_{L^{2}}.
\end{eqnarray*}
By Lemma \ref{pp}, we have
\begin{eqnarray}
\lefteqn{\||\nabla|^{s}(\langle\sqrt{3}\xi_{1}\rangle ue^{iS(u)})\|_{L^{2}}
}\nonumber\\
&\le&
C(\||\nabla|^{s-1}e^{iS(u)}\|_{L^{\frac{2s}{s-1}}}
+\||\nabla|^{s-1}(S'(u)ue^{iS(u)})\|_{L^{\frac{2s}{s-1}}})
\nonumber\\
& &
\qquad\times(
\|\nabla(\langle\sqrt{3}\xi_{1}\rangle u)\|_{L^{2s}}
+
\|[\nabla,\langle\sqrt{3}\xi_{1}\rangle] u\|_{L^{2s}}
)\nonumber\\
& &+C(1+\|S'(u)u\|_{L^{\infty}})
(\|\langle\sqrt{3}\xi_{1}\rangle u\|_{\dot{H}^{s}}
+
\|[\nabla,\langle\sqrt{3}\xi_{1}\rangle ]u\|_{\dot{H}^{s-1}}
).\label{a1}
\end{eqnarray}
Lemma \ref{pp}, Lemma \ref{ff} (ii) and the interpolation inequality 
yield
\begin{eqnarray}
\||\nabla|^{s-1}e^{iS(u)}\|_{L^{\frac{2s}{s-1}}}
&\le&C\|u\|_{L^{\infty}}^{p-1-\frac{s-1}{\sigma}}
\||\nabla|^{\sigma}u
\|_{L^{\frac{2s}{\sigma}}}^{\frac{s-1}{\sigma}}
\nonumber\\
&\le&C\|u\|_{L^{\infty}}^{p-2+\frac{1}{s}}
\|u\|_{\dot{H}^{s}}^{1-\frac1s},
\nonumber\\
&\le&C\|\langle\sqrt{3}\xi_{1}\rangle u\|_{L^{\infty}}^{p-2+\frac{1}{s}}
\|\langle\sqrt{3}\xi_{1}\rangle u\|_{H^{s}}^{1-\frac1s},
\label{a2}
\end{eqnarray}
where $\sigma$ satisfies $(s-1)/(p-1)<\sigma<1$.
Lemma \ref{ff} and (\ref{a2}) imply 
\begin{eqnarray}
\lefteqn{\||\nabla|^{s-1}(S'(u)ue^{iS(u)})\|_{L^{\frac{2s}{s-1}}}}
\nonumber\\
&\le&
C(\||\nabla|^{s-1}(S'(u)u)\|_{L^{\frac{2s}{s-1}}}
+\|S'(u)u\|_{L^{\infty}}
\||\nabla|^{s-1}e^{iS(u)}\|_{L^{\frac{2s}{s-1}}})
\nonumber\\
&\le&C(1+\|u\|_{L^{\infty}}^{p-1})
\|u\|_{L^{\infty}}^{p-2+\frac{1}{s}}
\|u\|_{\dot{H}^{s}}^{1-\frac1s}
\nonumber\\
&\le&C(1+\|\langle\sqrt{3}\xi_{1}\rangle u\|_{L^{\infty}}^{p-1})
\|\langle\sqrt{3}\xi_{1}\rangle u\|_{L^{\infty}}^{p-2+\frac{1}{s}}
\|\langle\sqrt{3}\xi_{1}\rangle u\|_{H^{s}}^{1-\frac1s}.
\label{a3}
\end{eqnarray}
The interpolation inequality yields
\begin{eqnarray}
\|\nabla(\langle\sqrt{3}\xi_{1}\rangle u)\|_{L^{2s}}
\le C\|\langle\sqrt{3}\xi_{1}\rangle u\|_{L^{\infty}}^{1-\frac1s}
\|\langle\sqrt{3}\xi_{1}\rangle u\|_{\dot{H}^{s}}^{\frac1s}.
\label{a4}
\end{eqnarray}
Since $[\nabla,\langle\sqrt{3}\xi_{1}\rangle]=3\xi_{1}
\langle\sqrt{3}\xi_{1}\rangle^{-1}$ is a  
pseudo-differential operator of order zero, by the $L^{p}$ boundedness of 
pseudo-differential operator (see \cite[Chapter VI]{Stein} for instance) 
and the interpolation inequality, 
we obtain
\begin{eqnarray}
\|[\nabla,\langle\sqrt{3}\xi_{1}\rangle] u\|_{L^{2s}}
&\le& C\|u\|_{L^{2s}}\le C\|u\|_{L^{\infty}}^{1-\frac1s}
\|u\|_{\dot{H}^{s}}^{\frac1s},
\label{a5}\\
\|[\nabla,\langle\sqrt{3}\xi_{1}\rangle] u\|_{\dot{H}^{s-1}}
&\le& C\|u\|_{H^{s-1}}\le C\|u\|_{H^{s}}.
\label{a6}
\end{eqnarray}
Combining (\ref{a1}), (\ref{a2}), (\ref{a3}), 
(\ref{a4}), (\ref{a5}) and (\ref{a6}), 
we have 
\begin{eqnarray*}
\||\nabla|^{s}(\langle\sqrt{3}\xi_{1}\rangle ue^{iS(u)})\|_{L^{2}}
\le C(1+\|\langle\sqrt{3}\xi_{1}\rangle u\|_{L^{\infty}}^{2p-2})
\|\langle\sqrt{3}\xi_{1}\rangle u\|_{H^{s}}.
\end{eqnarray*}

Next we consider the case $1<s<2\le p$. 
As in the previous case, 
we obtain (\ref{a1}).
Lemma \ref{pp}, Lemma \ref{ff} (i) and the interpolation inequality 
yield
\begin{eqnarray}
\||\nabla|^{s-1}e^{iS(u)}\|_{L^{\frac{2s}{s-1}}}
&\le&C\|u\|_{L^{\infty}}^{p-2}
\||\nabla|^{s-1}u
\|_{L^{\frac{2s}{s-1}}}
\nonumber\\
&\le&C\|u\|_{L^{\infty}}^{p-2+\frac{1}{s}}
\|u\|_{\dot{H}^{s}}^{1-\frac1s}
\nonumber\\
&\le&C\|\langle\sqrt{3}\xi_{1}\rangle u\|_{L^{\infty}}^{p-2+\frac{1}{s}}
\|\langle\sqrt{3}\xi_{1}\rangle u\|_{H^{s}}^{1-\frac1s}.
\label{b2}
\end{eqnarray}
Lemma \ref{ff} and (\ref{b2}) imply 
\begin{eqnarray}
\lefteqn{\||\nabla|^{s-1}(S'(u)ue^{iS(u)})\|_{L^{\frac{2s}{s-1}}}}
\nonumber\\
&\le&
C(\||\nabla|^{s-1}(S'(u)u)\|_{L^{\frac{2s}{s-1}}}
+\|S'(u)u\|_{L^{\infty}}
\||\nabla|^{s-1}e^{iS(u)}\|_{L^{\frac{2s}{s-1}}})
\nonumber\\
&\le&C(1+\|u\|_{L^{\infty}}^{p-1})
\|u\|_{L^{\infty}}^{p-2+\frac{1}{s}}
\|u\|_{\dot{H}^{s}}^{1-\frac1s}
\nonumber\\
&\le&C(1+\|\langle\sqrt{3}\xi_{1}\rangle u\|_{L^{\infty}}^{p-1})
\|\langle\sqrt{3}\xi_{1}\rangle u\|_{L^{\infty}}^{p-2+\frac{1}{s}}
\|\langle\sqrt{3}\xi_{1}\rangle u\|_{H^{s}}^{1-\frac1s}.
\label{b3}
\end{eqnarray}
Combining (\ref{a1}), (\ref{a4}), (\ref{a5}), 
(\ref{a6}), (\ref{b2}) and (\ref{b3}), we have 
\begin{eqnarray*}
\||\nabla|^{s}(\langle\sqrt{3}\xi_{1}\rangle ue^{iS(u)})\|_{L^{2}}
\le C(1+\|\langle\sqrt{3}\xi_{1}\rangle u\|_{L^{\infty}}^{2p-2})
\|\langle\sqrt{3}\xi_{1}\rangle u\|_{H^{s}}.
\end{eqnarray*}

Finally let us consider the case $2\le s<p<3$. 
Since 
\begin{eqnarray*}
\lefteqn{\Delta(\langle\sqrt{3}\xi_{1}\rangle ue^{iS(u)})}\\
&=&\Delta (\langle\sqrt{3}\xi_{1}\rangle u)(1+iS'(u)u)e^{iS(u)}\\
& &+\nabla u\cdot \nabla(\langle\sqrt{3}\xi_{1}\rangle u)
(2iS'(u)+iS''(u)u-S'(u)^{2}u)e^{iS(u)},\\
& &-iS'(u)u([\Delta,\langle\sqrt{3}\xi_{1}\rangle]u)e^{iS(u)}\\
& &-\nabla u([\nabla,\langle\sqrt{3}\xi_{1}\rangle]u)
(iS''(u)u-S'(u)^{2}u)e^{iS(u)},
\end{eqnarray*}
we obtain
\begin{eqnarray*}
\lefteqn{\||\nabla|^{s}(\langle\sqrt{3}\xi_{1}\rangle 
ue^{iS(u)})\|_{L^{2}}}\\
&\simeq&\||\nabla|^{s-2}\Delta(\langle\sqrt{3}\xi_{1}\rangle 
ue^{iS(u)})\|_{L^{2}}\\
&\le&
C(\||\nabla|^{s-2}
\{\Delta (\langle\sqrt{3}\xi_{1}\rangle u)e^{iS(u)}\}\|_{L^{2}}
+\||\nabla|^{s-2}\{\Delta (\langle\sqrt{3}\xi_{1}\rangle u) 
S'(u)ue^{iS(u)}\}\|_{L^{2}})\\
& &+C\||\nabla|^{s-2}\{\nabla u\cdot \nabla(\langle\sqrt{3}
\xi_{1}\rangle u)S'(u)e^{iS(u)}\}\|_{L^{2}}\\
& &+C\||\nabla|^{s-2}\{\nabla u\cdot \nabla(\langle\sqrt{3}
\xi_{1}\rangle u)S'(u)^{2}ue^{iS(u)}\}\|_{L^{2}}\\
& &+C\||\nabla|^{s-2}\{([\Delta,\langle\sqrt{3}\xi_{1}\rangle]u)
S'(u)ue^{iS(u)}\}\|_{L^{2}}\\
& &+C\||\nabla|^{s-2}\{\nabla u([\nabla,\langle\sqrt{3}\xi_{1}\rangle]u)
S''(u)ue^{iS(u)}\}\|_{L^{2}}\\
& &+C\||\nabla|^{s-2}\{\nabla u([\nabla,\langle\sqrt{3}\xi_{1}\rangle]u)
S'(u)^{2}ue^{iS(u)}\}\|_{L^{2}}.
\end{eqnarray*}
Hence by Lemma \ref{pp}, we have
\begin{eqnarray}
\lefteqn{\||\nabla|^{s}(\langle\sqrt{3}\xi_{1}\rangle ue^{iS(u)})\|_{L^{2}}}
\nonumber\\
&\le&
C(\||\nabla|^{s-2}e^{iS(u)}\|_{L^{\frac{2s}{s-2}}}
+\||\nabla|^{s-2}(S'(u)ue^{iS(u)})\|_{L^{\frac{2s}{s-2}}})
\nonumber\\
& &\qquad\times(\|\Delta (\langle\sqrt{3}\xi_{1}\rangle u)\|_{L^{s}}
+
\|[\Delta,\langle\sqrt{3}\xi_{1}\rangle]u\|_{L^{s}})\nonumber\\
& &+C(1+\|S'(u)u\|_{L^{\infty}})
(\|\langle\sqrt{3}\xi_{1}\rangle u\|_{\dot{H}^{s}}
+\|[\Delta,\langle\sqrt{3}\xi_{1}\rangle]u\|_{\dot{H}^{s-2}})
\nonumber\\
& &
+C\{\||\nabla|^{s-1}u\|_{L^{\frac{2s}{s-1}}}
(\|\nabla (\langle\sqrt{3}\xi_{1}\rangle u)\|_{L^{2s}}
+\|[\nabla,\langle\sqrt{3}\xi_{1}\rangle ]u)\|_{L^{2s}})\nonumber\\
& &
\qquad+\|\nabla u\|_{L^{2s}}
(\||\nabla|^{s-1} (\langle\sqrt{3}\xi_{1}\rangle u)\|_{L^{\frac{2s}{s-1}}}
+\||\nabla|^{s-2}[\nabla,\langle\sqrt{3}\xi_{1}\rangle]u)\|_{L^{\frac{2s}{s-1}}})\}
\nonumber\\
& &\qquad\qquad\times(\|S'(u)\|_{L^{\infty}}+\|S'(u)^{2}u\|_{L^{\infty}})
\nonumber\\
& &
+C\|\nabla u\|_{L^{2s}}
(\|\nabla(\langle\sqrt{3}\xi_{1}\rangle u)\|_{L^{2s}}
+
\|[\nabla,\langle\sqrt{3}\xi_{1}\rangle] u\|_{L^{2s}})
\nonumber\\
& &\qquad\times(\||\nabla|^{s-2}(S'(u)e^{iS(u)})\|_{L^{\frac{2s}{s-2}}}
+\||\nabla|^{s-2}(S'(u)^{2}ue^{iS(u)})\|_{L^{\frac{2s}{s-2}}}).
\label{c1}
\end{eqnarray}
Lemma \ref{ff} (i) and the interpolation inequality 
yield
\begin{eqnarray}
\||\nabla|^{s-2}e^{iS(u)}\|_{L^{\frac{2s}{s-2}}}
&\le&C\|u\|_{L^{\infty}}^{p-2}
\||\nabla|^{s-2}u
\|_{L^{\frac{2s}{s-2}}}
\nonumber\\
&\le&C\|u\|_{L^{\infty}}^{p-2+\frac{2}{s}}
\|u\|_{\dot{H}^{s}}^{1-\frac2s}
\nonumber\\
&\le&C\|\langle\sqrt{3}\xi_{1}\rangle u\|_{L^{\infty}}^{p-2+\frac{2}{s}}
\|\langle\sqrt{3}\xi_{1}\rangle u\|_{H^{s}}^{1-\frac2s}.
\label{c2}
\end{eqnarray}
Lemma \ref{ff} and (\ref{c2}) imply 
\begin{eqnarray}
\lefteqn{\||\nabla|^{s-1}(S'(u)ue^{iS(u)})\|_{L^{\frac{2s}{s-1}}}}
\nonumber\\
&\le&
C(\||\nabla|^{s-1}(S'(u)u)\|_{L^{\frac{2s}{s-1}}}
+\|S'(u)u\|_{L^{\infty}}
\||\nabla|^{s-1}e^{iS(u)}\|_{L^{\frac{2s}{s-1}}})
\nonumber\\
&\le&C(1+\|u\|_{L^{\infty}}^{p-1})
\|u\|_{L^{\infty}}^{p-2+\frac{1}{s}}
\|u\|_{\dot{H}^{s}}^{1-\frac1s}
\nonumber\\
&\le&C(1+\|\langle\sqrt{3}\xi_{1}\rangle u\|_{L^{\infty}}^{p-1})
\|\langle\sqrt{3}\xi_{1}\rangle u\|_{L^{\infty}}^{p-2+\frac{1}{s}}
\|\langle\sqrt{3}\xi_{1}\rangle u\|_{\dot{H}^{s}}^{1-\frac1s}.
\label{c3}
\end{eqnarray}
By the interpolation 
inequality, we obtain
\begin{eqnarray*}
\|\Delta (\langle\sqrt{3}\xi_{1}\rangle u)\|_{L^{s}}
&\le& C\|\langle\sqrt{3}\xi_{1}\rangle u\|_{L^{\infty}}^{1-\frac2s}
\|\langle\sqrt{3}\xi_{1}\rangle u\|_{\dot{H}^{s}}^{\frac2s},\\
\|\nabla (\langle\sqrt{3}\xi_{1}\rangle u)\|_{L^{2s}}
&\le& C\|\langle\sqrt{3}\xi_{1}\rangle u\|_{L^{\infty}}^{1-\frac1s}
\|\langle\sqrt{3}\xi_{1}\rangle u\|_{\dot{H}^{s}}^{\frac1s},\\
\||\nabla|^{s-1} (\langle\sqrt{3}\xi_{1}\rangle u)\|_{L^{\frac{2s}{s-1}}}
&\le& C\|\langle\sqrt{3}\xi_{1}\rangle u\|_{L^{\infty}}^{\frac1s}
\|\langle\sqrt{3}\xi_{1}\rangle u\|_{\dot{H}^{s}}^{1-\frac1s}.
\end{eqnarray*}
Since $[\Delta,\langle\sqrt{3}\xi_{1}\rangle]$ and 
$[\nabla,\langle\sqrt{3}\xi_{1}\rangle]$ are $1$st and $0$-th order 
pseudo-differential operators, by the $L^{p}$ boundness of 
pseudo-differential operator and the interpolation inequality, 
we obtain
\begin{eqnarray*}
\|[\Delta,\langle\sqrt{3}\xi_{1}\rangle] u\|_{L^{s}}
&\le& C\|\langle\nabla\rangle u\|_{L^{s}}
\le C\|\langle\Delta\rangle u\|_{L^{s}}\\
&\le& C\|\langle\sqrt{3}\xi_{1}\rangle u\|_{L^{\infty}}^{1-\frac2s}
\|\langle\sqrt{3}\xi_{1}\rangle u\|_{H^{s}}^{\frac2s},\\
\|[\Delta,\langle\sqrt{3}\xi_{1}\rangle]u\|_{\dot{H}^{s-2}}
&\le&C\|u\|_{H^{s-1}}\le C\|u\|_{H^{s}}\\
&\le&C\|\langle\sqrt{3}\xi_{1}\rangle u\|_{H^{s}},\\
\|[\nabla,\langle\sqrt{3}\xi_{1}\rangle ]u\|_{L^{2s}}
&\le&C\|u\|_{L^{2s}}\\
&\le&C\|\langle\sqrt{3}\xi_{1}\rangle u\|_{L^{\infty}}^{1-\frac1s}
\|\langle\sqrt{3}\xi_{1}\rangle u\|_{H^{s}}^{\frac1s},\\
\||\nabla|^{s-2} ([\nabla,\langle\sqrt{3}\xi_{1}\rangle] u)\|_{L^{\frac{2s}{s-1}}}
&\le&C\|\langle\nabla\rangle^{s-2}u\|_{L^{\frac{2s}{s-1}}}
\le C\|\langle\nabla\rangle^{s-1}u\|_{L^{\frac{2s}{s-1}}}\\
&\le&C\|\langle\sqrt{3}\xi_{1}\rangle u\|_{L^{\infty}}^{\frac1s}
\|\langle\sqrt{3}\xi_{1}\rangle u\|_{H^{s}}^{1-\frac1s},
\end{eqnarray*}
Combining the above inequalities with 
(\ref{c1}), (\ref{c2}), (\ref{c3}), 
we have
\begin{eqnarray*}
\||\nabla|^{s}(\langle\sqrt{3}\xi_{1}\rangle ue^{iS(u)})\|_{L^{2}}
\le C(1+\|\langle\sqrt{3}\xi_{1}\rangle u\|_{L^{\infty}}^{2p-2})
\|\langle\sqrt{3}\xi_{1}\rangle u\|_{H^{s}}.
\end{eqnarray*} 
This completes the proof of Lemma \ref{nl}. 
\end{proof}

\begin{proof}[Proof of Proposition \ref{nl1}.] 
The proof follows by applying Lemma \ref{nl} 
for $u=\langle\sqrt{3}\xi_{1}\rangle^{-1}\hat{\psi}_{+}$. 
\end{proof}


\section{Proof of Theorems \ref{nonlinear1} and \ref{nonlinear2}.} \label{sec:nonlinear}

In this section we prove Theorems \ref{nonlinear1} and \ref{nonlinear2}. 
Let $u_{+}$ be given by (\ref{u}). 
We first rewrite (\ref{FSP2}) as the integral equation. 
By Propositions \ref{linear} and \ref{nl1}, we have
\begin{eqnarray}
u(t)-u_+(t)=u(t)-W(t){{{{\mathcal F}}}}^{-1}w+R_1(t),\label{71}
\end{eqnarray}
where 
\begin{eqnarray}
w(t,\xi)=\hat{\psi}_+(\xi)e^{iS_+(t,\xi)} \label{72}
\end{eqnarray}
and $R_1$ satisfies 
\begin{eqnarray}
\|R_1(t)\|_{L^{\infty}(t,\infty;L_x^2)}&\le&
Ct^{-\beta}\|{{\mathcal F}}^{-1}w\|_{H^{0,s}}\nonumber\\
&\le&Ct^{-\beta}P(\|\psi_{+}\|_{H_{x}^{0,s}}),\label{78}\\
\|\langle\pt_{x_{1}}\rangle^{\frac{2}{qd}}R_1(t)\|_{L^q(t,\infty;L_x^{r})}&\le&
Ct^{-\frac{1}{q}-\beta}\|{{\mathcal F}}^{-1}w\|_{H^{0,s}}
\nonumber\\
&\le&
Ct^{-\frac{1}{q}-\beta}
P(\|\psi_{+}\|_{H_{x}^{0,s}})\label{781}
\end{eqnarray}
for $0<\beta<1/2$ and $s>d/2-(d-1)/r+1$, where $P$ is a polynomial 
without constant terms. Then (\ref{71}) is rewritten as follows:
\begin{eqnarray}
u(t)-u_+(t)=W(t){{{{\mathcal F}}}}^{-1}[{{{{\mathcal F}}}}W(-t)u-w]
+R_1(t).\label{73}
\end{eqnarray}
Let ${{\mathcal L}}=i\pt_t+(1/2)\Delta-(1/4)\pt_{x_{1}}^4$. 
From (\ref{FSP2}) and (\ref{72}), we obtain
\begin{eqnarray}
i\pt_t({{{{\mathcal F}}}}W(-t)u)&=&
{{{{\mathcal F}}}}W(-t){{{\mathcal L}}}u=\lambda
{{{{\mathcal F}}}}W(-t)|u|^{p-1}u,\label{74}\\
i\pt_tw&=&
\lambda
\frac{t^{-\frac{p-1}{2}d}}
{(3\xi_{1}^{2}+1)^{\frac{p-1}{2}}}|\hat{\psi}_{+}(\xi)|^{p-1}
\hat{\psi}_+(\xi)e^{iS_{+}(t,\xi)}.\label{75}
\end{eqnarray}
Subtracting (\ref{75}) from (\ref{74}), we have
\begin{eqnarray}
\lefteqn{i\pt_t({{{{\mathcal F}}}}W(-t)u-w)}\nonumber\\
&=&\lambda{{{{\mathcal F}}}}W(-t)|u|^{p-1}u
-\lambda
\frac{t^{-\frac{p-1}{2}d}}
{(3\xi_{1}^{2}+1)^{\frac{p-1}{2}}}|\hat{\psi}_{+}(\xi)|^{p-1}
\hat{\psi}_+(\xi)e^{iS_{+}(t,\xi)}.\label{76}
\end{eqnarray}
Proposition \ref{linear} yields 
\begin{eqnarray}
\lefteqn{
W(t){{{\mathcal F}}}^{-1}
[\lambda
\frac{t^{-\frac{p-1}{2}d}}
{(3\xi_{1}^{2}+1)^{\frac{p-1}{2}}}|\hat{\psi}_{+}(\xi)|^{p-1}
\hat{\psi}_+(\xi)e^{iS_{+}(t,\xi)}]}\qquad\qquad\qquad\qquad\qquad{}\nonumber
\\
&=&\lambda|u_+|^{p-1}u_++R_2(t),
\label{77}
\end{eqnarray}
where $R_2$ satisfies 
\begin{eqnarray}
\|R_2\|_{L^1(t,\infty;L_x^2)}\le
Ct^{-\beta}P(\|\psi_{+}\|_{H_{x}^{0,s}}),\label{79}
\end{eqnarray}
where $0<\beta<1/2$. 
Substituting (\ref{77}) into (\ref{76}), we obtain 
\begin{eqnarray*}
i\pt_t({{{{\mathcal F}}}}W(-t)u-w)
=\lambda{{{{\mathcal F}}}}W(-t)[|u|^{p-1}u-|u_+|^{p-1}u_+]
-{{{{\mathcal F}}}}W(-t)R_2.
\end{eqnarray*}
Integrating the above equation with respect to $t$ variable on $(t,\infty)$,
we have
\begin{eqnarray}
\lefteqn{u(t)-
W(t){{{\mathcal F}}}^{-1}w}\nonumber\\
&=&i\lambda\int_t^{+\infty}
W(t-\tau)[|u|^{p-1}u-|u_+|^{p-1}u_+](\tau)d\tau\nonumber\\
& &-i\int_t^{+\infty}W(t-\tau)R_2(\tau)d\tau.\label{80}
\end{eqnarray}
Combining (\ref{71}) with (\ref{80}), we obtain the following 
integral equation:
\begin{eqnarray}
u(t)-u_+(t)&=&
i\lambda\int_t^{+\infty}
W(t-\tau)[|u|^{p-1}u-|u_+|^{p-1}u_+](\tau)d\tau\nonumber\\
& &+R_1(t)-i\int_t^{+\infty}W(t-\tau)R_2(\tau)d\tau.\label{81}
\end{eqnarray}

We prove the existence of a solution to (\ref{81}). We first consider the case $d=3$. 
Notice that in this case the end point Strichartz estimate is available.

\vskip1mm
\noindent
{\bf Case: $d=3$.} To show the existence of $u$ satisfying (\ref{81}), we 
shall prove that the map
\begin{eqnarray}
\ \ \ \Phi[u](t)&=&u_+(t)+
i\lambda\int_t^{+\infty}
W(t-\tau)[|u|^{p-1}u-|u_+|^{p-1}u_+](\tau)d\tau
\nonumber\\
& &+R_1(t)-i\int_t^{+\infty}W(t-\tau)R_2(\tau)d\tau
\label{phi}
\end{eqnarray}
is a contraction on
\begin{eqnarray*}
{{{\bf X}}}_{\rho,T}&=&
\{u\in C([T,\infty);L^2(\rre^{3}))\cap \langle\pt_{x_{1}}
\rangle^{-\frac{2}{3q}}L_{loc}^{q}(T,\infty;L^{r}(\rre^{3}));\\
& &\qquad\qquad\qquad\qquad\qquad\qquad\qquad\qquad\qquad
\|u-u_+\|_{{{\bf X}}_T}\le \rho\},\\
\|v\|_{{{\bf X}}_T}
&=&\sup_{t\ge T}
t^{\alpha}(\|v\|_{L^{\infty}(t,\infty;L_x^2)}
+\|\langle\pt_{x_{1}}\rangle^{\frac{2}{3q}}v\|_{L^{q}(t,\infty;L_x^{r})})
\end{eqnarray*}
for some $T\ge3$, where 
\begin{eqnarray*}
\left(q,r\right)
=\left(\frac{4}{3p-5},\frac{6}{8-3p}\right). 
\end{eqnarray*}

Let $v(t)=u(t)-u_+(t)$ and 
$v\in{{\bf X}}_{\rho,T}$. Then
\begin{eqnarray*}
\lefteqn{\Phi[u](t)-u_+(t)}\\
&=&
i\lambda\int_t^{+\infty}
W(t-\tau)
[(|v+u_{+}|^{p-1}(v+u_{+})-|u_{+}|^{p-1}u_{+})](\tau)d\tau\\
& &+R_1(t)-i\int_t^{+\infty}W(t-\tau)R_2(\tau)d\tau.
\end{eqnarray*}
Since 
\begin{eqnarray*}
\lefteqn{||v+u_{+}|^{p-1}(v+u_{+})-|u_{+}|^{p-1}u_{+}|}\\
&\le& p\int_{0}^{1}||\theta v+u_{+}|^{p-1}-|u_{+}|^{p-1}|d\theta|v|+
p|u_{+}|^{p-1}|v|\\
&\le& C(|v|^{p-1}+|u_{+}|^{p-1})|v|,
\end{eqnarray*}
the Strichartz estimate (Lemma \ref{S}) implies
\begin{eqnarray}
\lefteqn{\|\Phi[u]-u_+
\|_{L^{\infty}(t,\infty;L_x^2)}
+\|\langle\pt_{x_{1}}\rangle^{\frac{2}{3q}}(\Phi[u]-u_+)\|_{L^q(t,\infty;L_x^{r})}}\nonumber\\
&\le&C(\||v|^{p-1}v\|_{L^{2}(t,\infty;L_x^{\frac65})}+\||u_{+}|^{p-1}v\|_{L^1(t,\infty;L_x^2)}\nonumber\\
& &+\|R_1\|_{L^{\infty}(t,\infty;L_x^2)}
+\|\langle\pt_{x_{1}}\rangle^{\frac{2}{3q}}R_1\|_{L^q(t,\infty;L_x^{r})}+
\|R_2\|_{L^1(t,\infty;L_x^2)}).\nonumber\\
\label{82}
\end{eqnarray}
By the H\"{o}lder inequality, we find
\begin{eqnarray*}
\||v|^{p-1}v\|_{L^{2}(t,\infty;L_x^{\frac65})}
&\le&C\|\|v\|_{L_x^{2}}^{p-1}\|v\|_{L_x^{r}}\|_{L^{2}(t,\infty)}
\nonumber\\
&\le&C\rho^{p-1}\|t^{-(p-1)\alpha}
\|v\|_{L_x^{r}}\|_{L^{2}(t,\infty)}\nonumber\\
&\le&C\rho^{p-1}\|t^{-(p-1)\alpha}\|_{L^{s}(t,\infty)}
\|v\|_{L^{q}((t,\infty);L^{r})}\nonumber\\
&\le& C\rho^p t^{-\alpha p+1-\frac{3}{4}(p-1)},\nonumber\\
\||u_{+}|^{p-1}v\|_{L^1(t,\infty;L_x^2)}
&\le&\|\|u_+\|_{L_x^{\infty}}^{p-1}\|v\|_{L_x^2}\|_{L^1(t,\infty)}\\
&\le& C\rho\|t^{-\frac{3}{2}(p-1)-\alpha}\|_{L^1(t,\infty)}\\
&\le& C\rho t^{-\frac{3}{2}(p-1)-\alpha+1},
\end{eqnarray*}
where $1/s=(7-3p)/4$. 
Substituting above two inequalities, (\ref{78}), 
(\ref{781}) and (\ref{79}) 
into (\ref{82}), we have
\begin{eqnarray*}
\|\Phi[u]-u_+\|_{{{\bf X}}_T}
\le C(\rho^{p}T^{-\alpha(p-1)+1-\frac{3}{4}(p-1)}+\rho T^{-\frac{3}{2}(p-1)+1}
+T^{\alpha-\beta}).
\end{eqnarray*}
Choosing $1/(p-1)-3/4<\alpha<\beta<1/2$ and  
$T$ large enough, 
we guarantee that $\Phi$ is a map onto ${{\bf X}}_{\rho,T}$. 
In a similar way, we can conclude that $\Phi$ is a contraction 
map on ${{\bf X}}_{\rho,T}$. Therefore, by 
Banach fixed point theorem 
one infers  that $\Phi$ has a unique fixed point in ${{\bf X}}_{\rho,T}$ 
which is the solution to the final state problem (\ref{FSP2}). 

Next, we show that the solution to (\ref{FSP2}) with   
finite ${{\bf X}}_T$ norm is unique. Let 
$u_{1}$ and $u_{2}$ be two solutions satisfying 
$\|u_{1}\|_{{{\bf X}}_T}<\infty$ and $\|u_{2}\|_{{{\bf X}}_T}<\infty$. 
We put $t_1=\inf\{t\in[T,\infty);u_{1}(s)=u_{2}(s)$ for any $s\in[t,\infty)\}$ 
and $\rho=\max\{\|u_{1}\|_{{{\bf X}}_T},\|u_{2}\|_{{{\bf X}}_T}\}$. 
If $t_1=T$, then $u_{1}(t)=u_{2}(t)$ on $[T,\infty)$ which is desired 
result. If $T<t_1$, as in (\ref{82}) 
by the Strichartz inequality (Lemma \ref{S}), 
we have
\begin{eqnarray*}
\lefteqn{\|\langle\pt_{x_{1}}\rangle^{\frac{2}{3q}}(u_{1}-u_{2})\|_{L^q(t_0,t_1,L_x^r)}}\\
&\le&C\rho^{p-1}(t_0^{1-s(p-1)\alpha}-t_1^{1-s(p-1)\alpha})^{1/s}
\|\langle\pt_{x_{1}}\rangle^{\frac{2}{3q}}(u_{1}-u_{2})\|_{L^q(t_0,t_1,L_x^r)},
\end{eqnarray*}
for $t_0\in[T,t_1)$. Since $1-s(p-1)\alpha<0$, 
we can choose $t_0\in[T,t_1)$ so that 
$C\rho^{p-1}(t_0^{1-s(p-1)\alpha}-t_1^{1-s(p-1)\alpha})^{1/s}<1$. Then
$\|\langle\pt_{x_{1}}\rangle^{\frac{2}{3q}}(u_{1}-u_{2})
\|_{L^q(t_0,t_1,L_x^r)}\le0$
which implies that $u_{1}(t)\equiv u_{2}(t)$ on 
$[t_0,t_1]$. This contradicts the assumption 
of $t_1$. Hence $u_{1}(t)=u_{2}(t)$ on $[T,\infty)$.

From (\ref{FSP2}), 
we obtain
\begin{eqnarray}
u(t)=W(t-T)u(T)-i\lambda\int_T^tW(\tau)|u|^{p-1}u(\tau)d\tau.
\label{5.19}
\end{eqnarray}
Since $u(T)\in L_x^{2}(\rre^{3})$, 
combining the argument by \cite{Tsutsumi} 
with the Strichartz estimate (Lemma \ref{S}) 
and $L^2$ conservation law for (\ref{5.19}), 
we can prove that (\ref{5.19}) has a unique 
global solution in $C(\rre;L_x^2(\rre^{3}))\cap
\langle\pt_{x_{1}}\rangle^{-2/(3q)}L_{loc}^q(\rre;L_x^{r}(\rre^{3}))$. 
Therefore the solution $u$ of (\ref{FSP2}) can be extended to all times. 

Finally we show that the solution to (\ref{FSP2}) converges to 
$W(t)\psi_{+}$ in $L^{2}$ as $t\to\infty$. 
Let 
\begin{eqnarray*}
u_{+}^{0}(t,x)=\frac{t^{-\frac{3}{2}}}{\sqrt{3\xi_{1}^{2}+1}}\hat{\psi}(\mu)
e^{\frac34it\mu_{1}^{4}+\frac12it|\mu|^{2}-i\frac34\pi},
\end{eqnarray*}
where $\mu$ is given by (\ref{mu}). 
Then by $u-u_{+}\in{{\bf X}}_T$ and Proposition \ref{linear}, 
we have
\begin{eqnarray*}
\lefteqn{\|u(t)-W(t)\psi_{+}\|_{L_{x}^{2}}}\\
&\le&
\|u(t)-u_{+}\|_{L_{x}^{2}}
+\|u_{+}-u_{+}^{0}\|_{L_{x}^{2}}
+\|u_{+}^{0}-W(t)\psi_{+}\|_{L_{x}^{2}}\\
&\le&Ct^{-\alpha}+\|S_{+}(t,\mu)\|_{L_{x}^{\infty}}\|u_{+}\|_{L_{x}^{2}}
+Ct^{-\beta}P(\|\psi_{+}\|_{H^{0,s}})\\
&\le&Ct^{-\alpha}+Ct^{1-\frac32(p-1)}P(\|\psi_{+}\|_{H^{0,s}})
+Ct^{-\beta}\|\psi_{+}\|_{H^{0,s}}\\
&\le&t^{-\gamma}P(\|\psi_{+}\|_{H^{0,s}})
\end{eqnarray*}
where $1/(p-1)-3/4<\alpha<1/2$, $3/8<\beta<1/2$ 
and $\min\{1/(p-1)-3/4,1-3(p-1)/2,3/8\}<\gamma$. 
This completes the proof of Theorem \ref{nonlinear2}. 

\vskip1mm
\noindent
{\bf Case: $d=2$.} Next we show Theorem \ref{nonlinear1}. 
In this case the end point Strichartz estimate is not available. 
Instead, we use the admissible pair which is close to the end point. 

To show the existence of $u$ satisfying (\ref{81}), we 
shall prove that the map $\Phi$ given by (\ref{phi}) 
is a contraction on
\begin{eqnarray*}
{{{\bf X}}}_{\rho,T}&=&
\{u\in C([T,\infty);L^2(\rre^{2}))\cap \langle\pt_{x_{1}}
\rangle^{-\frac{1}{q}}L_{loc}^{q}(T,\infty;L^{r}(\rre^{2}));\\
& &\qquad\qquad\qquad\qquad\qquad\qquad\qquad\qquad\qquad
\|u-u_+\|_{{{\bf X}}_T}\le \rho\},\\
\|v\|_{{{\bf X}}_T}
&=&\sup_{t\ge T}
t^{\alpha}(\|v\|_{L^{\infty}(t,\infty;L_x^2)}
+\|\langle\pt_{x_{1}}\rangle^{\frac{1}{q}}v\|_{L^{q}(t,\infty;L_x^{r})})
\end{eqnarray*}
for some $T\ge3$, where 
\begin{eqnarray*}
\left(q,r\right)
=\left(\frac{2}{p-2+2\varepsilon},\frac{2}{3-p-2\varepsilon}\right). 
\end{eqnarray*}

Let $v(t)=u(t)-u_+(t)$ and $v\in{{\bf X}}_{\rho,T}$. 
Then 
the Strichartz estimate (Lemma \ref{S}) implies
\begin{eqnarray}
\lefteqn{\|\Phi[u]-u_+
\|_{L^{\infty}(t,\infty;L_x^2)}
+\|\langle\pt_{x_{1}}\rangle^{\frac{1}{q}}(\Phi[u]-u_+)\|_{L^q(t,\infty;L_x^{r})}}\nonumber\\
&\le&C(\||v|^{p-1}v\|_{L^{\frac{2}{1+2\varepsilon}}(t,\infty;L_x^{\frac{1}{1-\varepsilon}})}
+\||u_{+}|^{p-1}v\|_{L^1(t,\infty;L_x^2)})\nonumber\\
& &+C(\|R_1\|_{L^{\infty}(t,\infty;L_x^2)}
+\|\langle\pt_{x_{1}}\rangle^{\frac{1}{q}}R_1\|_{L^q(t,\infty;L_x^{r})}+
\|R_2\|_{L^1(t,\infty;L_x^2)}).\nonumber\\
\label{821}
\end{eqnarray}
By the H\"{o}lder inequality, 
\begin{eqnarray*}
\||v|^{p-1}v\|_{L^{\frac{2}{1+2\varepsilon}}(t,\infty;L_x^{\frac{1}{1-\varepsilon}})}
&\le&C\|\|v\|_{L_x^{2}}^{p-1}\|v\|_{L_x^{r}}\|_{L^{\frac{2}{1+2\varepsilon}}(t,\infty)}
\nonumber\\
&\le&C\rho^{p-1}\|t^{-(p-1)\alpha}
\|v\|_{L_x^{r}}\|_{L^{\frac{2}{1+2\varepsilon}}(t,\infty)}\nonumber\\
&\le&C\rho^{p-1}\|t^{-(p-1)\alpha}\|_{L^{s}(t,\infty)}
\|v\|_{L^{q}((t,\infty);L^{r})}\nonumber\\
&\le& C\rho^p t^{-\alpha p+1-\frac{1}{2}(p-1)},\nonumber\\
\||u_{+}|^{p-1}v\|_{L^1(t,\infty;L_x^2)}
&\le&\|\|u_+\|_{L_x^{\infty}}^{p-1}\|v\|_{L_x^2}\|_{L^1(t,\infty)}\\
&\le& C\rho\|t^{-(p-1)-\alpha}\|_{L^1(t,\infty)}\\
&\le& C\rho t^{-(p-1)-\alpha+1},
\end{eqnarray*}
where $1/s=(3-p)/2$. 
Substituting above two inequalities, (\ref{78}), 
(\ref{781}) and (\ref{79}) 
into (\ref{821}), we have
\begin{eqnarray*}
\|\Phi[u]-u_+\|_{{{\bf X}}_T}
\le C(\rho^{p}T^{-\alpha(p-1)+1-\frac{1}{2}(p-1)}+\rho T^{-(p-1)+1}
+T^{\alpha-\beta}).
\end{eqnarray*}
Choosing $1/(p-1)-1/2<\alpha<\beta<1/2$ and  
$T$ large enough, 
we guarantee that $\Phi$ is a map onto ${{\bf X}}_{\rho,T}$. 
In a similar way we can conclude that $\Phi$ is a contraction 
map on ${{\bf X}}_{\rho,T}$. Therefore, by Banach fixed point theorem 
one finds  that $\Phi$ has a unique fixed point in ${{\bf X}}_{\rho,T}$ 
which is the solution to the final state problem (\ref{FSP2}). 
Since the rest part of the proof is same as the proof of 
Theorem \ref{nonlinear2}, we omit the detail. 
This completes the proof of Theorem \ref{nonlinear1}. $\qed$

\section{Final comments}

Finally, we give several comments.  

\vskip2mm
\noindent
1. As mentioned in the introduction, we restricted our study of (\ref{D}) 
to the case where  $\alpha,\beta$ and $\gamma$ satisfy $(\alpha+\frac{3\beta^{2}}{8\gamma})\alpha>0$ 
and $\alpha\gamma<0.$ If those conditions are violated, one should replace (\ref{4NLS}) by the "non-elliptic" NLS
\begin{eqnarray}\label{nonell}
i\partial_tu +\frac{1}{2}\Delta_\perp-\frac{1}{2}\partial_{x_1}^2 u\pm\frac{1}{4}\partial_{x_1}^4u=\lambda| u|^{p-1}u\qquad t>0,\ x\in\rre^{d}.
\end{eqnarray}
It is very likely that similar scattering results hold  true in this case since the linear estimates should be essentially the same.

\vskip2mm
\noindent
2. We were concerned in this paper with scattering issues for equation (1.1). On a different regime, one might look for a possible blow-up of solutions in the focusing case ($\lambda <0$ in (1.2)). It is conjectured in \cite{FIS} that a finite time blow-up should arise when 
$p\geq 1+8/(2d-1).$ Proving this fact is an interesting open question.

\vskip3mm
\noindent {\bf Acknowledgments.} 
The part of this work was done while J.S was 
visiting at Department of Mathematics at 
Universit\'e de Paris-Sud, Orsay 
whose hospitality they gratefully acknowledge.
J.S. is partially supported by JSPS, Strategic Young Researcher Overseas
Visits Program for Accelerating Brain Circulation and by JSPS,
Grant-in-Aid for Scientific Research (B) 17H02851.


\begin{thebibliography}{A}
\bibitem{ADRT} 
Aceves A.B., De Angelis C., Rubenchik A.M. and Turitsyn S.K., 
\textit{Multidimensional solitons in fiber arrays},
Optical  Letters {\bf 19} (1995), 329-331.

\bibitem{AHN}
Aoki K., Hayashi N. and Naumkin P.I., \textit{
Global existence of small solutions for the fourth-order 
nonlinear Schr\"odinger equation}. 
NoDEA Nonlinear Differential Equations Appl. 
{\bf 23} (2016), Art. 65, 18 pp.

\bibitem{BKS}
Ben-Artzi M., Koch H. and Saut J.-C., \textit{
Dispersion estimates for fourth order Schr\"{o}dinger equations}. 
C. R. Acad. Sci. Paris S\'{e}r. I Math. {\bf 330} (2000), 87--92.

\bibitem{Ber} Berg\'e L.,
\textit{Wave  collapse in  physics :  principles and  applications  to  light and  plasma
waves},
Phys. Rep., {\bf 303} (1998), 259.

\bibitem{BCGJ} Bonheure D., Castera J.-B., Gou T. and Jeanjean L., 
\textit{
Strong instability of ground states to a fourth order Schr\"{o}dinger equation}, 
preprint available at arXiv:1703.07977v2 (2017).

\bibitem{B}
Bouchel O., \textit{Remarks on NLS with higher order anisotropic dispersion}. 
Adv. Differential Equations {\bf 13} (2008), 169--198. 

\bibitem{BL} Boulenger T. and  Lenzmann E., \textit{
Blowup for biharmonic NLS}, Annales Scientifiques de l' ENS, 
{\bf 50} (3) (2017), 503-544.

\bibitem{Caz}
Cazenave T., \textit{``Semilinear Schr\"{o}dinger 
equations"}. 
Courant Lecture Notes in Mathematics, {\bf 10}.
American Mathematical Society (2003).

\bibitem{CW} 
Christ F.M. and Weinstein M.I., 
\textit{Dispersion of small amplitude solutions of 
the generalized Korteweg-de Vries equation}. 
J. Funct. Anal. {\bf 100} (1991) 87--109.

\bibitem{FI} Fibich G. and Ilan B.,
\textit{Optical light bullets in a pure Kerr medium},
Optics letters, {\bf 29}
(2004), 887-889.

\bibitem{FIS1} Fibich Gadi., Ilan B. and Papanicolaou G.,
\textit{Self-focusing with fourth-order dispersion}. 
SIAM J. Appl. Math. {\bf 62} (2002), 1437--1462.

\bibitem{FIS}
Fibich G., Ilan B. and Schochet S.,
\textit{
Critical  exponents  and  collapse  of  nonlinear  Schr\"{o}dinger  equations  with  anisotropic  fourth-order  dispersion},
Nonlinearity, {\bf 16}
(2003), 1809-1821.



 

\bibitem{HHW}
Hao C., Hsiao L. and Wang B., \textit{
Well-posedness of Cauchy problem for the fourth order nonlinear 
Schr\"odinger equations in multi-dimensional spaces}. 
J. Math. Anal. Appl. {\bf 328} (2007), 58--83. 

\bibitem{HN7} Hayashi N., Mendez-Navarro Jesus A. 
and Naumkin P. I.,
\textit{Scattering of solutions to the fourth-order nonlinear 
Schr\"odinger equation}. Commun. Contemp. Math. {\bf 18} (2016), 
1550035, 24 pp.

\bibitem{HN8} Hayashi N., Mendez-Navarro Jesus A. and Naumkin P. I.,
\textit{Asymptotics for the fourth-order nonlinear Schr\"odinger 
equation in the critical case}. J. Differential Equations 
{\bf 261} (2016), 5144--5179.
 
\bibitem{HN3} Hayashi N. and Naumkin P. I., 
\textit{Domain and range of the modified wave operator for 
Schr\"{o}dinger equations with a critical nonlinearity}. 
Comm. Math. Phys. 267 (2006), no. 2, 477--492.

\bibitem{HN1} Hayashi N. and Naumkin P. I., 
\textit{Asymptotic properties of solutions to dispersive equation of 
Schr\"{o}dinger type}. J. Math. Soc. Japan {\bf 60} (2008), 631--652.

\bibitem{HN4} Hayashi N. and Naumkin P. I., 
\textit{Large time asymptotics for the fourth-order nonlinear 
Schr\"odinger equation}. J. Differential Equations {\bf 258} 
(2015), 880--905. 

\bibitem{HN5} Hayashi N. and Naumkin P. I., 
\textit{Global existence and asymptotic behavior of solutions 
to the fourth-order nonlinear Schr\"odinger equation in the critical 
case}. Nonlinear Anal. {\bf 116} (2015), 112--131.

\bibitem{HN9} Hayashi N. and Naumkin P. I., 
\textit{On the inhomogeneous fourth-order nonlinear Schr\"odinger 
equation}. J. Math. Phys. {\bf 56} (2015), 093502, 25 pp.

\bibitem{HN6} Hayashi N. and Naumkin P. I., 
\textit{Factorization technique for the fourth-order nonlinear 
Schr\"dinger equation}. Z. Angew. Math. Phys. {\bf 66} 
(2015), 2343--2377. 

\bibitem{HO} 
Hirayama H. and Okamoto M., \textit{
Well-posedness and scattering for fourth order nonlinear 
Schr\"odinger type equations at the scaling critical regularity}. 
Commun. Pure Appl. Anal. {\bf 15} (2016), 831--851.

\bibitem{K} Karpman V. I., \textit{Stabilization of soliton instabilities by higher-order dispersion: 
fourth order nonlinear Schr\"{o}dinger-type equations}. 
Phys. Rev. E {\bf 53} (2) (1996) R1336-R1339.

\bibitem{KT} Keel M and Tao T., \textit{
Endpoint Strichartz estimates}. 
Amer. J. Math. {\bf 120} (1998), 955--980.

\bibitem{KM}
Kenig C.E. and Merle F., \textit{Global well-posedness, scattering and blow-up for 
the energy-critical, focusing, non-linear Schr\"{o}dinger equation in the radial case}. 
Invent. Math. {\bf 166} (2006), 645--675.

\bibitem{KPV} Kenig C. E., Ponce G. and Vega L., 
\textit{Oscillatory integrals and regularity of 
dispersive equations}. Indiana Univ. math J. 
{\bf 40} (1991), 33-69.



\bibitem{MXZ1} Miao C., Xu G. and Zhao L., 
\textit{Global well-posedness and scattering for the focusing energy-critical 
nonlinear Schr\"{o}dinger equations of fourth order in the radial case}. 
J. Differential Equations {\bf 246} (2009), 3715--3749.

\bibitem{MXZ2} Miao C., Xu G. and Zhao L., 
\textit{Global wellposedness and scattering for the defocusing energy-critical 
nonlinear Schr\"{o}dinger equations of fourth order in dimensions $d\geq9$}. 
J. Differetial Equations {\bf 251} (2011), 3381-3402.

\bibitem{MZ}
Miao C. and Zheng J., 
\textit{Scattering theory for the defocusing fourth-order 
Schr\"odinger equation}. Nonlinearity {\bf 29} (2016), 692--736.

\bibitem{O} Ozawa T., 
\textit{Long range scattering for nonlinear 
Schr\"{o}dinger equations in one space dimension}. 
Comm. Math. Phys. {\bf 139} (1991), 479--493. 

\bibitem{P0} Pausader B., \textit{
The focusing energy-critical fourth-order Schr\"odinger 
equation with radial data}. 
Discrete Contin. Dyn. Syst. {\bf 24} (2009), 1275--1292.

\bibitem{P1} Pausader B., \textit{Global well-posedness for energy critical 
fourth-order Schr\"{o}dinger equations in the radial case}. 
Dyn. Partial Differ. Equ. {\bf 4} (2007), 197--225.

\bibitem{P2} Pausader B., \textit{
The cubic fourth-order Schr\"{o}dinger equation}. 
J. Funct. Anal., {\bf 256} (2009), 2473--2517.


\bibitem{PX}
Pausader B. and Xia S., \textit{
Scattering theory for the fourth-order Schr\"odinger 
equation in low dimensions}. 
Nonlinearity {\bf 26} (2013), 2175--2191.

\bibitem{S} Segata J., \textit{A remark on asymptotics of solutions to 
Schr\"{o}dinger equation with fourth order dispersion}. 
Asymptotic Analysis, {\bf 75} (2011), 25--36.

\bibitem{Stein} Stein E.M., \textit{
``Harmonic analysis: real-variable methods, orthogonality, and oscillatory integrals''}. 
Princeton University Press (1993).

\bibitem{Tsutsumi} Tsutsumi Y.,
 \textit{$L^2$-solutions for nonlinear Schr\"{o}dinger equations and nonlinear groups}. 
 Funkcial. Ekvac. {\bf 30} (1987), 115--125.

\bibitem{V} Visan M., \textit{The defocusing energy-critical 
nonlinear Schr\"odinger equation in dimensions five and higher}. 
Ph.D. Thesis. UCLA (2007). 

\bibitem{WF} Wen S. and Fan D., 
\textit{Spatiotemporal instabilities in nonlinear Kerr media in the presence of arbitrary higher
order dispersions}, J. Opt. Soc. Am. B {\bf 19} (2002), 1653-1659.
\end{thebibliography}
\end{document}